\documentclass[11pt]{amsart}
\usepackage[colorlinks=true]{hyperref}
\usepackage{amsmath}
\usepackage{amssymb}
\usepackage{graphicx}
\usepackage{xcolor}
\usepackage{mathtools}
\newtheorem{theorem}{Theorem}[section]

\theoremstyle{definition}

\newtheorem{remark}[theorem]{Remark}

\numberwithin{equation}{section}
\newcommand{\ul}[1]{\underline{#1}}

\title[Jordan--Moore--Gibson--Thompson equation]{Vanishing relaxation time limit of the Jordan--Moore--Gibson--Thompson wave equation with Neumann and absorbing boundary conditions}

\author[Barbara Kaltenbacher]{Barbara Kaltenbacher}
\address[B. Kaltenbacher]{Institute of Mathematics, Alpen-Adria-Universit\"at Klagenfurt, Universit\"atsstra\ss e 65-67, 9020 Klagenfurt, Austria}
\email{{barbara.kaltenbacher@aau.at}}

\author[Vanja Nikoli\'c]{Vanja Nikoli\'c}
\address[V. Nikoli\'c]{Chair of Numerical Mathematics, Department of Mathematics, Technical University of Munich, 
	Boltzmannstra\ss e 3, 85748 Garching, Germany}
\email{\tt vanja.nikolic@ma.tum.de}

\keywords{third order in time PDE, energy method, singular limit}

\subjclass[2010]{35L72, 35L77, 35L80, 35B40, 49K20, 49Q10}

\begin{document}

\begin{abstract}
We study the Jordan--Moore--Gibson--Thompson (JMGT) equation, a third order in time wave equation that models nonlinear sound propagation, in the practically relevant setting of Neumann and absorbing boundary conditions. In the analysis, we pay special attention to dependencies on the coefficient $\tau$ of the third order time derivative that plays the physical role of relaxation time. We establish local in time well-posedness and derive energy bounds that can be made independent of $\tau$ under appropriate conditions. This fact allows us to pass to the limit $\tau\to0$ and recover solutions of a classical model in nonlinear acoustics, the Westervelt equation, as singular limits of solutions to the JMGT equation.
\end{abstract}

\maketitle

\section{Introduction}
Driven by applications, particularly of high-intensity ultrasound~\cite{kennedy2003high, wu2001pathological, yoshizawa2009high}, the field of modeling and analysis of nonlinear acoustics has recently found much interest. In this paper, we analyze a nonlinear third order in time acoustic wave equation that has been put forward in~\cite{JordanMaxwellCattaneo14,JordanMaxwellCattaneo09} and studied, along with its linearization, in~\cite{DellOroPata,KLM12_MooreGibson,KLP12_JordanMooreGibson,kaltenbacher2019jordan,LasieckaWang15b, LasieckaWang15a,
LiuTriggiani13,MarchandMcDevittTriggiani12,PellicerSolaMorales}.

\subsection{Problem setting and modeling}
A classical model of nonlinear sound propagation is the Westervelt equation~\cite{Westervelt63}
\begin{equation} \label{Westervelt0}
\psi_{tt}-c^2\Delta \psi - \delta\Delta \psi_t = k(\psi_t^2)_t,
\end{equation}
where $c>0$ is the speed of sound, $\delta>0$ the diffusivity of sound, $k$ a parameter quantifying the nonlinearity of the equation, and $\psi$ the acoustic velocity potential. Acoustic potential is related to the acoustic particle velocity $\vec{v}$ via $\vec{v}=-\nabla\psi$, and to the acoustic pressure $p$ via $p=\rho_0\psi_t$, where $\rho_0$ denotes the mean mass density. \\
\indent To overcome the infinite signal speed paradox which is unnatural in wave propagation, Fourier's law is replaced in the derivation of acoustic models by the Maxwell-Cattaneo law; cf.~\cite{JordanMaxwellCattaneo14}. This action leads to  a third order in time equation containing the (finite) relaxation time as a parameter $\tau$:
\begin{equation} \label{WesterveltMC_0}
\tau\psi_{ttt}+\psi_{tt}-c^2\Delta \psi - b\Delta \psi_t = k(\psi_t^2)_t, 
\end{equation}
known as the Jordan--Moore--Gibson--Thompson equation (of Westervelt type, i.e., containing no gradient nonlinearities), where 
\begin{equation}\label{b}
b=\delta+\tau c^2\,.
\end{equation}
In nonlinear acoustics, excitation is commonly achieved by an array of piezoelectric transducers; see~\cite{kaltenbacher2007numerical}. We thus employ inhomogeneous Neumann conditions
\[
\frac{\partial \psi}{\partial n}=g
\]
on some surface $\Gamma$.
The wave equations will be considered in a bounded $C^{1,1}$ domain $\Omega$, with $\Gamma\subseteq\partial\Omega$, motivated from the point of view of applications by the need to restrict attention (as well as numerical computations) to a certain domain of interest even though wave propagation in reality occurs in free space. Working on a bounded domain is also crucial from an analysis point of view since it enables the use of certain embedding results that would not be valid on unbounded domains. \\
\indent This reasoning necessitates the use of appropriate boundary conditions to avoid spurious reflections of the outgoing waves on the boundary of the domain of interest $\Omega$, which we here do by imposing linear absorbing boundary conditions on the rest of the boundary
\[
\frac{\partial \psi}{\partial n}=-\beta \psi_t \quad \mbox{ on }\Sigma=\partial\Omega\setminus\Gamma\,,
\]
where $\beta>0$ is a fixed positive coefficient; see~\cite{Nikolic15} and the references therein. Moreover, we confine ourselves to the setting of homogeneous initial conditions, which is practically relevant in applications such as lithotripsy~\cite{yoshizawa2009high}. The results can be extended to nonhomogeneous initial conditions in a straightforward manner. \\
\indent Some of the first steps into well-posedness and long-time behavior of the Westervelt equation have been made in a joint paper  \cite{KL09Westervelt} by Irena Lasiecka and one of the authors of this paper. One of her key observations that enabled this analysis was the fact that \eqref{Westervelt0} can be formulated as a second-order strongly damped wave equation  
\begin{equation} \label{Westervelt}
(1-2k \psi_t)\psi_{tt}-c^2\Delta \psi - \delta\Delta \psi_t = 0 \quad \mbox{ in }\Omega\times(0,T) 
\end{equation}
with a nonlinear coefficient $(1-2k \psi_t)$ of the second time derivative. The positivity and non-degeneracy of this factor is crucial for the mathematical analysis as well as for the physical validity of the model.\\
\indent The strong damping term $-\delta\Delta \psi_t$ allows to estimate the $L^\infty(0,T;H^2(\Omega))$ norm of $\psi_t$ and therewith, by virtue of the embedding $H^2(\Omega)\to L^\infty(\Omega)$, to guarantee nondegeneracy of \eqref{Westervelt} for small initial and boundary data. As a downside, this term renders the equation parabolic -- its linearization gives rise to an analytic semigroup \cite{KL09Westervelt} and to maximal parabolic regularity \cite{MW11} -- and thus leads to the infinite speed of propagation.
An analogous reformulation can be done for the JMGT equation
\begin{equation} \label{WesterveltMC}
\tau\psi_{ttt}+(1-2k \psi_t)\psi_{tt}-c^2\Delta \psi - b\Delta \psi_t = 0 \quad \mbox{ in }\Omega\times(0,T), 
\end{equation}
where additional challenges arise due to the appearance of a third order in time derivative. As desired from a physical point of view, this term counteracts the strong damping and mathematically leads to a loss of analyticity of the semigroup as well as maximal parabolic regularity; see~\cite[Remark 1.3]{KLM12_MooreGibson},~\cite[Subsection 6.2.1]{MarchandMcDevittTriggiani12}, and~\cite{LeCroneSimonett}.
\subsubsection{A relaxed JMGT equation}\hspace*{0.15cm} As an alternative to enforcing non-degeneracy by means of higher order estimates, we also introduce a relaxation of the JMGT equation for which we will prove existence of a less regular solution:
\begin{equation} \label{WesterveltMC_relaxed}
\begin{aligned}
\tau\psi_{ttt}+h(\psi_t)\psi_{tt}-c^2\Delta \psi - b\Delta \psi_t = 0 \mbox{ in }\Omega\times(0,T),
\end{aligned}
\end{equation}
where the function $h \in C^0(\mathbb{R})$ is assumed to be bounded:
\begin{align} \label{conditions_h}
 \underline{\alpha} \leq h(s) \leq \overline{\alpha}, \quad \forall s \in \mathbb{R}.
\end{align}
Such an approach to modeling is often taken, e.g., in the analysis of predictive tumor models to control the triple product terms while having $H^1$ regular solutions; see~\cite{garcke2017well}. In practice, we might choose the function $h$ as
\begin{align*}
h(s)=1-\min \{-1,\, \max \{1,\, 2k s\} \}
\end{align*}
since $h(\psi_t)=1-2k \psi_t$ a.e. if $2k\|\psi_t\|_{L^\infty(\Omega \times (0,T))} < 1$.
\subsubsection{Linearized JMGT equation} To establish well-posedness of \eqref{WesterveltMC} and \eqref{WesterveltMC_relaxed}, we also study the following linearization of these equations:
\begin{equation} \label{WesterveltMC_lin}
\begin{aligned}
\tau\psi_{ttt}+\alpha(x,t)\psi_{tt}-c^2\Delta \psi - b\Delta \psi_t = f(x,t) \mbox{ in }\Omega\times(0,T),
\end{aligned}
\end{equation}
which is sometimes called the Stokes--Moore--Gibson--Thompson (SMGT) equation~\cite{bucci2018feedback}.
\medskip

We note that this paper is a follow-up to \cite{kaltenbacher2019jordan}, where we have studied the JMGT equation and its singular limit as $\tau\to0$ in the simpler setting of homogeneous Dirichlet boundary conditions. The purpose of the present paper is to treat the practically relevant situation of Neumann and absorbing boundary conditions, which indeed turns out to require different energy estimates, as well as additional considerations concerning higher spatial regularity of solutions $\psi_t(t)\in H^s(\Omega)$ for $s>\frac32$ with the possibly mixed boundary conditions. The latter is crucial for avoiding degeneracy, i.e., guaranteeing positivity of the coefficient $1-2k\psi_t$ in \eqref{WesterveltMC}, via the embedding $H^s(\Omega)\to L^\infty(\Omega)$.
As an alternative to the high spatial regularity enforced for this purpose in previous publications on \eqref{Westervelt}, \eqref{WesterveltMC}, and other models of nonlinear acoustics, we also consider the relaxed version \eqref{WesterveltMC_relaxed}, for which we establish well-posedness with weaker spatial regularity.\\

\indent The remainder of this paper is organized as follows. We first investigate the pure Neumann-case setting $\Gamma=\partial\Omega$. To this end, in Section \ref{sec:West_lin}, we analyze the linearized equation \eqref{WesterveltMC_lin} on three different levels of assumptions and regularity results. Firstly, assuming $\alpha\in L^\infty$ without any sign condition and $f\in L^2$, which gives well-posedness with $H^1$ regularity in space and a $\tau$-dependent energy bound. Secondly, assuming additionally $\alpha$ to be positive and bounded away from zero, which renders the energy bound $\tau$-independent. The third case includes additional stronger regularity assumptions on $f$ and $\alpha$, which yields $H^2$ regularity in space, as needed to guarantee non-degeneracy, with a $\tau$-independent energy bound.\\
\indent For the nonlinear models under consideration here, we correspondingly show well-posedness of the relaxed JMGT equation \eqref{WesterveltMC_relaxed} in a low regularity regime without sign condition on $h$ (Section \ref{sec:West_relaxed_fixedpoint}) and of the original JMGT equation \eqref{WesterveltMC} in a higher regularity setting with a strictly positive coefficient $(1-2k\psi_t)$ (Section \ref{sec:West_fixedpoint}). The latter goes with a $\tau$-independent bound, which allows us to pass to the limit as $\tau\to0$ in Section \ref{sec:Singular_limit} and recover the classical Westervelt equation \eqref{Westervelt} as a singular limit of JMGT.
The final Section \ref{sec:ABC} deals with the situation of absorbing boundary conditions, i.e., the case when $\mbox{meas}(\partial\Omega\setminus\Gamma)>0$.
\subsubsection{Notation}
The time interval and the spatial domain are often omitted for notational simplicity when writing norms; for example, $\|\cdot\|_{L^p L^q}$ denotes the norm on $L^p(0,T;L^q(\Omega))$. We denote the $L^2(\Omega)$ inner product by $(\cdot,\cdot)_{L^2}$ and the $L^2(\Omega)$ norm as well as the absolute value by $|\cdot|$.
\section{Analysis of the linearized JMGT equation} \label{sec:West_lin}
We next focus on the analysis of the linearized JMGT equation \eqref{WesterveltMC_lin}
complemented with inhomogeneous Neumann data and zero initial conditions.
\subsection{$H^1$ regularity with a $\tau$-dependent bound} We begin by proving existence of an $H^1$ regular solution of \eqref{WesterveltMC_lin}. Note that here we do not impose any restrictions on the sign of the coefficient $\alpha$. However, as a downside, the bounds we will derive on the solution will not be uniform with respect to $\tau$.
\begin{theorem} \label{th:wellposedness_lin_lower_tau}
	Let $c^2$, $b$, $\tau>0$, and let $T>0$. Assume that
	\begin{itemize}
		\item $\alpha \in  L^\infty(0,T; L^\infty(\Omega))$,
        \smallskip
		\item $f\in L^2(0,T; L^2(\Omega))$, \smallskip
		\item $g \in H^2(0,T; H^{-1/2}(\Gamma))$,\smallskip
		\item $(g, g_t)\vert_{t=0}=(0, 0)$ (compatibility with inital data).
	\end{itemize}
	Then there exists a unique weak solution $\psi$ of the problem 
	\begin{equation} \label{ibvp_linear}
		\begin{aligned}
		\begin{cases}
		\tau\psi_{ttt}+\alpha(x,t)\psi_{tt}-c^2\Delta \psi - b\Delta \psi_t = f(x,t) \quad \mbox{ in }\Omega\times(0,T), \\[2mm]
		\dfrac{\partial \psi}{\partial n}=g \quad \mbox{ on } \Gamma\times(0,T),\\[2mm]
		(\psi, \psi_t, \psi_{tt})=(0, 0, 0) \quad \mbox{ in }\Omega\times \{0\},
		\end{cases}
		\end{aligned}
		\end{equation}
	in the weak $(H^{1})^{\star}$ sense that satisfies 
	\begin{equation*} 
	\begin{aligned}
	\psi \in \, W^{1, \infty}(0;T;H^1(\Omega)) \cap W^{2, \infty}(0,T; L^2(\Omega))\cap H^3(0,T;H^1(\Omega)^*).
	\end{aligned}
	\end{equation*}
Furthermore, the solution fullfils the estimate
	\begin{equation}\label{energy_est_lin}
	\begin{aligned}
& \tau^2 \|\psi_{ttt}\|_{L^2 (H^1)^\star}^2+\tau \|\psi_{tt}\|^2_{L^\infty L^2}+\|\psi_t\|^2_{L^\infty H^1}\\[1mm]
\leq&\, C(\alpha, \tau, T)\left(\|g\|^2_{W^{1, \infty}H^{-1/2 }}+\|g_{t}\|^2_{H^1 H^{-1/2}}+\|f\|^2_{L^2 L^2}\right). 
	\end{aligned}
	\end{equation}
The constant above is given by
\begin{equation} \label{C_Theorem2.1}
\begin{aligned} 
 & C(\alpha, \tau, T)\\
 =&\, \begin{multlined}[t]C_1 \left(\tfrac{1}{\tau^2}\|\alpha\|^2_{L^\infty L^\infty}+T^2+1 \right)\\ \times \textup{exp}\left(C_2 (\tfrac{1}{\tau}+\tfrac{1}{\tau}\|\alpha\|_{L^\infty L^\infty}+1+T)T\right)
(1+\tau), \end{multlined}
\end{aligned}
\end{equation}
where $C_1$, $C_2>0$ do not depend on $\tau, T$, or $\alpha$. 
\end{theorem}
\begin{proof}
We conduct the proof by employing Galerkin approximations in space and compactness arguments; cf.~\cite{EvansBook, Roubicek}.\\[3mm]
\noindent \textbf{Existence of a solution.} Let $\{w_i\}_{i \in \mathbb{N}}$ denote the eigenfunctions of the Neumann-Laplacian operator $- \Delta$:
\begin{equation} \label{eigenf_Laplacian}
\begin{aligned}
-\Delta w=&\, \lambda w \quad \mbox{ in }\Omega, \\
\frac{\partial w}{\partial n}=& \, 0 \quad  \mbox{ on }\Gamma.
\end{aligned}
\end{equation}
Then $\{w_i\}_{i \in \mathbb{N}}$ can be normalized to form an orthogonal basis of $H^1(\Omega)$ and to be orthonormal with respect to the $L^2(\Omega)$ scalar product.  \\
\indent We fix $n \in \mathbb{N}$ and introduce $V_n=\text{span}\{w_1, \ldots, w_n\}$. Our approximate solution is given by
\begin{equation}
\begin{aligned}
\psi^n(x,t)=& \, \displaystyle \sum_{i=1}^n \xi_i(t)w_i(x),\\
\end{aligned}
\end{equation}
where $\xi_i:(0,T) \rightarrow \mathbb{R}$, $i \in \{1,\ldots,n\}$. We then consider the following approximation of the original problem
\begin{equation} \label{ibvp_semi-discrete}
\begin{aligned} 
\begin{cases}
\hspace{3mm}(\tau \psi^n_{ttt}+\alpha \psi^n_{tt}, \phi)_{L^2}+(c^2 \nabla \psi^n+b \nabla \psi_t^n, \nabla \phi)_{L^2} \\[1mm]
= (f, \phi)_{L^2}+(c^2g+bg_t, \phi)_{L^2(\Gamma)}, \\[1mm]
\text{for every $\phi \in V_n$ pointwise a.e. in $(0,T)$}, \\[1mm]
(\psi^n(0), \psi_t^n(0), \psi^n_{tt}(0))=(0, 0, 0).
\end{cases}
\end{aligned}
\end{equation}
Let $I^n=[I_{ij}]$, $M^n=[M_{ij}]$, $K^n=[K_{ij}]$, $C^n=[C_{ij}]$, and $F^n=[F_{i}]$, where
\begin{equation}
\begin{aligned}
& I^n_{ij}=(w_i, w_j)_{L^2}=  \delta_{ij}, \ M^n_{ij}(t)=(\alpha w_i, w_j)_{L^2}, \\
& K^n_{ij}=(\nabla w_i, \nabla w_j)_{L^2},\\
& F^n_i=(f, w_i)_{L^2}+(c^2 g +bg_t, w_i)_{L^2(\Gamma)},
\end{aligned}
\end{equation}
and $\delta_{ij}$ denotes the Kronecker delta. By introducing $\xi^n=[\xi_1 \ldots \xi_n]^T$,  problem \eqref{ibvp_semi-discrete} can be rewritten as a system of ordinary differential equations:
\begin{equation} \label{ODE_system}
\begin{aligned}
\begin{cases}
\tau I^n \xi^n_{ttt}+M^n \xi^n_{tt}+b K^n \xi_t^n+c^2 K^n \xi^n=F^n(t), \\
(\xi^n(0), \xi^n_t(0), \xi^n_{tt}(0))=(0, 0, 0).
\end{cases}
\end{aligned}
\end{equation}
Existence of a solution $\xi^n \in H^3(0, T_n)$ of \eqref{ODE_system} can be then obtained from standard theory of ODEs; cf.~\cite[Chapter 1]{Roubicek}. Therefore, problem \eqref{ibvp_semi-discrete} has a solution $\psi^n \in H^3(0,T_n; V_n)$. \\[2mm]
\noindent \textbf{Energy estimate.} We next want to derive a bound for $\psi^n$ that is uniform with respect to $n$. To this end, we add the term $(\psi^n_t, \phi)$ to both sides of \eqref{ibvp_semi-discrete}, test the problem with $\phi=\psi^n_{tt}$, and integrate over $(0, t)$ to obtain
\begin{equation} \label{West_first_est_discrete}
\begin{aligned}
&\tfrac{\tau}{2} |\psi^n_{tt}(t)|^2_{L^2}+\tfrac{b}{2}|\nabla \psi^n_t(t)|^2_{L^2}+|\psi^n_{t}(t)|^2_{L^2}\\
\leq&\, c^2 \left|\int_0^t \int_{\Omega} \nabla \psi^n \cdot \nabla \psi^n_{tt} \, \textup{d}x \textup{d}s \right|+\|\alpha\|_{L^\infty L^\infty}\|\psi_{tt}\|^2_{L^2_tL^2}\\
&+\|\psi^n_t\|_{L^2_tL^2}\|\psi^n_{tt}\|_{L^2_tL^2}+ \|f\|_{L^2L^2} \|\psi^n_{tt}\|_{L^2_tL^2}\\
&+\int_{\Gamma}(c^2 g+b g_t)\psi^n_{tt} \, \textup{d}x \textup{d}s,
\end{aligned}
\end{equation}
since $\psi_{t}^n(0)=\psi_{tt}^n(0)=0$. To simplify the notation, we have omitted the argument $(s)$ under the time integral
and employed the abbreviation $L_t^2L^2$ for $L^2(0,t;L^2(\Omega))$. We can further estimate the terms on the right-hand side in \eqref{West_first_est_discrete} as follows
\begin{equation*}
\begin{aligned}
c^2 \left |\int_0^t \int_{\Omega} \nabla \psi^n \cdot \nabla \psi^n_{tt} \, \textup{d}x 
 \textup{d}s \right|=& \,  c^2 \left| \int_{\Omega} \nabla \psi^n(t) \cdot \nabla \psi^n_{t}(t) \, \textup{d}x - \|\nabla \psi^n_t\|^2_{L^2L^2} \right| \\
\leq& \, c^2 |\nabla \psi^n(t)|_{L^2}|\nabla \psi^n_{t}(t)|_{L^2}+c^2\|\nabla \psi^n_t\|^2_{L^2L^2}.
\end{aligned}
\end{equation*}
We estimate the boundary integral by first integrating by parts with respect to time and then employing H\"older's inequality and the trace theorem:
\begin{equation*}
\begin{aligned}
&\int_0^t\int_{\Gamma}(c^2 g+b g_t)\psi^n_{tt} \, \textup{d}x \textup{d}s \\[1mm]
=& \,\int_{\Gamma}(c^2 g(t)+b g_t(t))\psi^n_{t}(t) \, \textup{d}x-\int_0^t\int_{\Gamma}(c^2 g_{t}+b g_{tt})\psi^n_{t} \, \textup{d}x \textup{d}s \\[1mm]
\leq& \, |c^2 g(t)+b g_t(t)|_{H^{-1/2}}C_{tr}|\psi^n_t(t)|_{H^1}+\|c^2 g_{t}+b g_{tt}\|_{L^2 H^{-1/2}}C_{tr}\|\psi_t^n\|_{L^2_t H^1}.
\end{aligned}
\end{equation*}
We note that the regularity assumption on $g_{tt}$ is introduced since we do not want to involve the $H^1$ norm of $\psi^n_{tt}$ in the estimates. After employing these bounds in \eqref{West_first_est_discrete} as well as Young's $\varepsilon$-inequality with $\varepsilon \in \{b/8, 1/2, 1/4\}$, we arrive at
\begin{equation} \label{West_2nd_estimate_discrete}
\begin{aligned}
&\tfrac{\tau}{2} |\psi^n_{tt}(t)|^2_{L^2}+\tfrac{b}{2}|\nabla \psi^n_t(t)|^2_{L^2}+|\psi^n_{t}(t)|^2_{L^2}\\[1mm]
\leq&\, \tfrac {2c^4}{b} |\nabla \psi^n(t)|_{L^2}+\tfrac{b}{8}|\nabla \psi^n_{t}(t)|^2_{L^2}+c^2\|\nabla \psi^n_t\|^2_{L^2_tL^2}+\tfrac12 \|\psi^n_t\|^2_{L^2_tL^2}\\
&+(1+\|\alpha\|_{L^\infty L^\infty})\|\psi^n_{tt}\|^2_{L^2_tL^2}+\tfrac12 \|f\|^2_{L^2L^2}\\
&+C_{tr}^2|c^2 g(t)+b g_t(t)|^2_{H^{-1/2}}+\tfrac{1}{4}|\psi^n_t(t)|^2_{L^2}\\
&+C_{tr}^2 \,\tfrac{2}{b}|c^2 g(t)+b g_t(t)|^2_{H^{-1/2}}+\tfrac{b}{8}|\nabla \psi^n_t(t)|^2_{L^2}\\
&+\tfrac12 (C_{tr})^2\|c^2 g_{t}+b g_{tt}\|^2_{L^2 H^{-1/2}}+\tfrac12\|\psi_t^n\|^2_{L^2_t H^1}.
\end{aligned}
\end{equation}
Since $\psi^n(0)=0$, we can further estimate the first term on the right-hand side as
$$|\nabla \psi^n(t)|_{L^2} \leq \sqrt{T}\|\nabla \psi^n_t\|_{L^2 L^2}, $$
a.e. in time. Then an application of Gronwall's inequality to \eqref{West_2nd_estimate_discrete} and taking a supremum over $(0,T_n)$ leads to
\begin{equation} \label{discrete_est_0}
\begin{aligned}
&\tau \|\psi^n_{tt}\|^2_{L^\infty L^2}+\|\nabla \psi^n_t\|^2_{L^\infty L^2}+\|\psi^n_{t}\|^2_{L^\infty L^2}\\[1mm]
\leq&\, C(\alpha, \tau, T)(\|g\|^2_{W^{1, \infty}H^{-1/2 }}+\|g_{t}\|^2_{H^1 H^{-1/2}}+\|f\|^2_{L^2 L^2}),
\end{aligned}
\end{equation}
where the constant is given by
\begin{align} \label{const_}
\overline{C}(\alpha, \tau, T)=\overline{C}_1 \, \textup{exp}(\overline{C}_2 (\tfrac{1}{\tau}+\tfrac{1}{\tau}\|\alpha\|_{L^\infty L^\infty}+1+T)T)
(1+\tau),
\end{align}
and $\overline{C}_1$, $\overline{C}_2>0$ do not depend on $\tau$ or $n$. Since the right-hand side of \eqref{discrete_est_0} does not depend on $T_n$, we are allowed to extend the existence interval to $(0,T)$. \\
\indent Note that the (weak) $\tau$ dependence of the constant \eqref{const_} via the factor $1+\tau$ results from the $\tau$ dependence of $b$ according to \eqref{b}, while the left hand side of the equation is not affected by this due to the fact that $b\geq\delta$ holds for all $\tau\geq0$. \\
\indent Morover, we can obtain a bound on the third time derivative of $\psi^n$ by noting that
\begin{equation}
\begin{aligned}
&\left|\int_0^t \int_{\Omega} \tau \psi^n_{ttt} \xi \, \textup{d}x \textup{d}s\right| \\
\leq& \, \begin{multlined}[t]\left(\|\alpha\|_{L^\infty L^\infty}\|\psi^n_{tt}\|_{L^2L^2}+c^2\|\nabla \psi^n\|_{L^2L^2}+b \|\nabla \psi_t^n\|_{L^2L^2} \right.\\
+\left.\|f\|_{L^2L^2}+\|c^2g+bg_t\|_{L^2 H^{-1/2}} \right)\|\xi\|_{L^2 H^1},\end{multlined}
\end{aligned}
\end{equation}
for all $\xi \in L^2(0,T; H^1(\Omega))$. By also taking into account \eqref{discrete_est_0}, it follows that
\begin{equation} \label{est_psi_ttt}
\begin{aligned}
&\tau \|\psi^n_{ttt}\|_{L^2 (H^1)^\star}\\
\leq&\, \overline{\overline{C}}(\alpha, \tau, T)(\|g\|_{W^{1, \infty}H^{-1/2 }}+\|g_{t}\|_{H^1 H^{-1/2}}+\|f\|_{L^2 L^2}),
\end{aligned}
\end{equation}
where the constant is given by 
\begin{equation*}
\begin{aligned}
\overline{\overline{C}}(\alpha, \tau, T)=& \,
\overline{\overline{C}}_1 \left(\tfrac{1}{\tau}\|\alpha\|_{L^\infty L^\infty}+T+1 \right) \\
&\times\textup{exp}(\overline{\overline{C}}_2 (\tfrac{1}{\tau}+\tfrac{1}{\tau}\|\alpha\|_{L^\infty L^\infty}+1+T)T)
(1+\tau).
\end{aligned}
\end{equation*}
We can then combine estimates \eqref{est_psi_ttt} and \eqref{discrete_est_0} to get
\begin{equation} \label{discrete_est_1}
\begin{aligned}
&\tau^2 \|\psi^n_{ttt}\|^2_{L^2 (H^1)^\star}+\tau \|\psi^n_{tt}\|^2_{L^\infty L^2}+\|\nabla \psi^n_t\|^2_{L^\infty L^2}+\|\psi^n_{t}\|^2_{L^\infty L^2}\\[1mm]
\leq&\, C(\alpha, \tau, T)(\|g\|^2_{W^{1, \infty}H^{-1/2 }}+\|g_{t}\|^2_{H^1 H^{-1/2}}+\|f\|^2_{L^2 L^2}),
\end{aligned}
\end{equation}
where the constant is given by \eqref{C_Theorem2.1} and $C_1$, $C_2>0$ do not depend on $\tau$ or $n$.\\
\indent  Since the right-hand side of \eqref{discrete_est_1} is independent of $n$, we can find a subsequence, denoted again by $\{\psi^n\}_{n \in \mathbb{N}}$, and a function $\psi$ such that
\begin{equation*} 
\begin{alignedat}{4} 
\psi_{ttt}^n  &\relbar\joinrel\rightharpoonup \psi_{ttt} &&\text{ weakly}  &&\text{ in } &&L^2(0,T; (H^1(\Omega))^\star),  \\
\psi_{tt}^n  &\relbar\joinrel\rightharpoonup \psi_{tt} &&\text{ weakly-$\star$}  &&\text{ in } &&L^\infty(0,T;L^2(\Omega)),  \\
\psi_t^n &\relbar\joinrel\rightharpoonup \psi_t &&\text{ weakly-$\star$} &&\text{ in } &&L^\infty(0,T; H^1(\Omega)).
\end{alignedat} 
\end{equation*}
It is then straightforward to show that $\psi$ solves \eqref{ibvp_linear} and fulfills the estimate \eqref{energy_est_lin}; cf.~\cite{kaltenbacher2019jordan}. 
\end{proof}
\subsection{$H^1$ regularity with a $\tau$-independent bound} We next prove a modification of the previous result on $H^1$ regularity with a bound on the solution that is uniform with respect to $\tau$ in a bounded interval $(0,\overline{\tau}]$.
\begin{theorem} \label{th:wellposedness_lin_lower} Let the assumption of Theorem~\ref{th:wellposedness_lin_lower} hold and assume additionally that 
for some fixed $\overline{\tau}>0$, $\tau\in(0,\overline{\tau}]$, as well as that
\begin{align} \label{non-degeneracy_condition}
\exists \ul{\alpha}>0: \,\alpha(t)\geq\ul{\alpha}\ \text{  a.e. in } \Omega \times (0,T).  
\end{align}
Then the solution of \eqref{ibvp_linear_hom} satisfies the estimate 
	\begin{equation}\label{energy_est_lin_tau_independent}
	\begin{aligned}
&\tau^2 \|\psi_{ttt}\|_{L^2 (H^1)^*}^2+\tau \|\psi_{tt}\|^2_{L^\infty L^2}+\|\psi_{tt}\|^2_{L^2L^2}+\|\psi_t\|^2_{L^\infty H^1}\\[1mm]
\leq&\, C(T) \left(\|g\|^2_{W^{1, \infty}H^{-1/2 }}+\|g_{t}\|^2_{H^1 H^{-1/2}}+\|f\|^2_{L^2 L^2}\right), 
	\end{aligned}
	\end{equation}
where the constant is given by
\begin{equation*}
\begin{aligned}
C(T)=C_3 \,(1+T^2) \textup{exp}(C_4 (1+T)T)
(1+\overline{\tau}),
\end{aligned}
\end{equation*}
and $C_3$, $C_4>0$ do not depend on $\tau, T$, or $\alpha$.
\end{theorem}
\begin{proof}
\indent The proof follows analogously to the proof of Theorem~\ref{th:wellposedness_lin_lower_tau}. However, in case the condition \eqref{non-degeneracy_condition} holds, estimate \eqref{West_first_est_discrete} can be replaced by 
\begin{equation} \label{West_first_est_discrete_alphapos}
\begin{aligned}
&\tau |\psi^n_{tt}(t)|^2_{L^2}+\underline{\alpha}\|\psi_{tt}\|^2_{L^2L^2}+\tfrac{b}{2}|\nabla \psi^n_t(t)|^2_{L^2}+|\psi^n_{t}(t)|^2_{L^2}\\
\leq&\, c^2 \left|\int_0^t \int_{\Omega} \nabla \psi^n \cdot \nabla \psi^n_{tt} \, \textup{d}x \textup{d}s \right|\\
&+\|\psi^n_t\|_{L^2L^2}\|\psi^n_{tt}\|_{L^2L^2}+ \|f\|_{L^2L^2} \|\psi^n_{tt}\|_{L^2L^2}\\
&+\int_{\Gamma}(c^2 g+b g_t)\psi^n_{tt} \, \textup{d}x \textup{d}s,
\end{aligned}
\end{equation}
from which we can derive \eqref{energy_est_lin_tau_independent} after some standard manipulations, first in a discrete setting before passing to the limit.
\end{proof}
\subsection{$H^2$ regularity with $\tau-$independent bound}
To be able to later show well-posedness for the JMGT equation \eqref{WesterveltMC}, we need $H^2$ regularity of the solution to the linearized equation \eqref{WesterveltMC_lin}. We therefore also prove a higher-order regularity result with an energy estimate that has a $\tau$-indepen\-dent right-hand side.  To this end, we first define an appropriate extension of the Neumann boundary data to the interior.\\

\noindent \textbf{Extension of the inhomogeneous boundary data.} Following~\cite{bucci2018feedback, kaltenbacher2011well}, for $h \in H^s(\Gamma)$, we introduce the harmonic extension operator $N: h \mapsto v$, where $v$ solves
\begin{equation} \label{defN}
\begin{aligned}
\begin{cases}
-\Delta v+v&=0 \quad \text{ in } \Omega, \\
\hspace*{1cm}\dfrac{\partial v}{\partial n}&=h \quad \text{ on } \Gamma=\partial \Omega,
\end{cases}
\end{aligned}
\end{equation}
which for negative $s$ we interpret in the variational sense 
\[
\begin{aligned} 
\langle \nabla v , \nabla \phi \rangle +\langle v , \phi \rangle = \langle h , \phi \rangle \quad
\text{for every $\phi \in H^1(\Omega)$}.
\end{aligned}
\]
It is known that the operator $N$ is a linear bounded mapping
\begin{align} \label{operator_N}
N: H^s(\partial \Omega) \rightarrow H^{s+3/2}(\Omega),
\end{align}
for $s \in \mathbb{R}$; see~\cite{kaltenbacher2011well, lasiecka1991regularity}. In the upcoming proof, we will employ the particular cases $s=-1/2$ and $s=1/2$ 
and denote the norm of $N$ in both cases by $C_N$. 
Furthermore, since we extend time-dependent Neumann data $g$ to the interior, we apply the mapping $N$ pointwise a.e. in time 
and denote the resulting operator by $N$ again, i.e., $(Ng)(t):=N g(t)$.
We note that due to the linearity of $N$, it holds that $ \partial_t (Ng)(t)=(Ng_t)(t)$.\\
\indent  We study the following initial-boundary value problem for $\bar{\psi}=\psi-Ng$ with homogeneous boundary data:
\begin{equation} \label{ibvp_linear_hom}
	\begin{aligned}
	\begin{cases}
 \tau\bar{\psi}_{ttt}+\alpha\bar{\psi}_{tt}-c^2\Delta \bar{\psi} - b\Delta \bar{\psi}_t\\[2mm]
 = \, f-\tau Ng_{ttt}-\alpha Ng_{tt}+c^2 \Delta Ng+b\Delta Ng_t \ \mbox{ in }\Omega\times(0,T), \\[2mm]
	\dfrac{\partial \bar{\psi}}{\partial n}=0 \quad \mbox{ on } \Gamma\times(0,T),\\[2mm]
	(\bar{\psi}, \bar{\psi}_t, \bar{\psi}_{tt})=(0, 0, 0) \quad \mbox{ in }\Omega\times \{0\},
	\end{cases}
	\end{aligned}
	\end{equation}
 provided that the compatibility conditions between the function $g$ and initial data stated below hold.
\begin{theorem} \label{th:wellposedness_lin_higher}
	Let $c^2$, $b>0$, $\tau \in (0, \overline{\tau})$, for some $\overline{\tau}>0$, and let $T>0$. Assume that
	\begin{itemize}
		\item $\alpha \in X_\alpha^W:=L^\infty(0,T; W^{1,3}(\Omega)\cap L^\infty(\Omega))$ ,\smallskip
		\item  $\ \exists \, \ul{\alpha}>0: \,\alpha(t)\geq\ul{\alpha}\ \text{  a.e. in } \Omega \times (0,T), $ \smallskip
		\item $f\in H^1(0,T; L^2(\Omega))$, \smallskip
		\item $		\|\nabla \alpha\|_{L^\infty L^3}<\underline{\alpha}/\left (6  C_{H^1,L^6} \right)$, \smallskip
  	\item $g \in H^3(0,T; H^{-1/2}(\Gamma))\cap W^{1,\infty}(0,T;H^{1/2}(\Gamma))$, \\[1mm] 
    $(g, g_t , g_{tt})\vert_{t=0}=(0, 0, 0)$.
	\end{itemize}
	Then there exists a unique weak solution $\psi$ of the problem  \eqref{ibvp_linear} such that
		\begin{equation*} 
	\begin{aligned}
	\psi \in \, X^W := \begin{multlined}[t] \, W^{1, \infty}(0,T; H^2(\Omega)) \cap W^{2, \infty}(0;T;H^1(\Omega)) \\ \cap H^3(0,T;L^2(\Omega)). \end{multlined}
	\end{aligned}
	\end{equation*}
	Moreover, the solution fullfils the estimate
\begin{equation} \label{energy_est_lin_higher}
\begin{aligned}
& \begin{multlined}[c]\tau^2 \|\psi_{ttt}\|_{L^2L^2}+\tau
\|\psi_{tt}\|^2_{L^\infty H^1}\\
+\|\psi_{tt}\|^2_{L^2 H^1}+\|- \Delta \psi_t\|^2_{L^\infty L^2} \end{multlined}\\[1mm]
\leq&\,C(\alpha, T)  \begin{multlined}[t]\left(\tau^2 \|g_{ttt}\|^2_{L^2 H^{-1/2}}+ \tau \|g_{tt}\|^2_{L^\infty H^{-1/2}}+ \|g\|^2_{H^2 H^{-1/2}}\right. \\ \left. +\|g_t\|^2_{L^{\infty} H^{1/2}}+ \| f\|^2_{H^1 L^2} \right).  \end{multlined}
\end{aligned}
\end{equation}
The constant above is given by
\begin{equation*}
\begin{aligned}
C(\alpha, T)=& \, \begin{multlined}[t] C_5 \,(\|\alpha\|^2_{L^\infty L^\infty}+\|\nabla \alpha\|^2_{L^\infty L^3}+1)\\
\times \textup{exp}\,\left(C_6 \, \left(\|\alpha\|^2_{L^\infty L^\infty}+\|\nabla \alpha\|^2_{L^\infty L^3} +1+T+T^2 \right) T  \right)
(1+\overline{\tau}), \end{multlined}
\end{aligned}
\end{equation*}
where $C_5$, $C_6>0$ do not depend on $n$ or $\tau$. 
\end{theorem}
\begin{proof}
The proof follows along the lines of~\cite[Theorem 4.1]{kaltenbacher2019jordan} by employing Galerkin approximations in space of the solution $\bar{\psi}$ to \eqref{ibvp_linear_hom} and compactness arguments, but with a modification of energy estimates due to the new terms related to the extension operator. The solution of the original problem \eqref{ibvp_linear} is obtained afterwards as $\psi=\bar{\psi}+Ng$. \\
\indent We use the eigenfunctions $\{w_i\}_{i \in \mathbb{N}}$ of the homogeneous Neumann-Laplacian as the basis of $H^2_{\textup{N}}(\Omega):=\{v \in H^2(\Omega)\left \vert \right. \, \tfrac{\partial v}{\partial n}=0 \ \text{on } 
\Gamma\}$ and an orthonormal basis of $L^2(\Omega)$; cf.~\cite{Math-01-00318}. 
The Galerkin approximation $\bar{\psi}^n$ of $\bar{\psi}$ is defined by 
\begin{equation} \label{ibvp_semi-discrete_H2}
\begin{aligned} 
&(\tau \bar{\psi}^n_{ttt}+\alpha \bar{\psi}^n_{tt}, \phi)_{L^2}+(c^2 \nabla \bar{\psi}^n+b \nabla \bar{\psi}_t^n, \nabla \phi)_{L^2} \\
=&\, (f-\tau Ng_{ttt}-\alpha Ng_{tt}+c^2 \Delta Ng+b\Delta Ng_t, \phi)_{L^2},
\end{aligned}
\end{equation}
for every $\phi \in V_n$ pointwise a.e. in $(0,T)$, with $(\bar{\psi}^n(0), \bar{\psi}_t^n(0), \bar{\psi}^n_{tt}(0))=(0, 0, 0).$ As before, the existence of a solution $\bar{\psi}^n \in H^3(0,T; V_n)$ for the semi-discretization of the problem in $V_n=\text{span}\{w_1, \ldots, w_n\}$ follows from the standard ODE existence theory; see, for example,~\cite[Chapter 1]{Roubicek}. We focus our attention on deriving the crucial energy estimate. \\

\noindent \textbf{Energy estimate.} We note that $\phi=-\Delta \bar{\psi}^n_{tt}$ belongs to $V_n$ since $\bar{\psi}^n_{tt}$ is a linear combination of eigenfunctions of the Laplacian. Testing the semi-discrete problem with $\phi=-\Delta \bar{\psi}^n_{tt}$ and integrating over $(0,t)$, where $t \leq T$, yields the energy identity 
\begin{equation}\label{enid1}
\begin{aligned}
& \tfrac{\tau}{2} |\nabla \bar{\psi}^n_{tt}(t)|^2+\| \sqrt{\alpha} \nabla \bar{\psi}^n_{tt}\|^2_{L_t^2L^2}
+\tfrac{b}{2}|-\Delta \bar{\psi}_t^n(t)|^2 \\
=&\,-\int_0^t(\bar{\psi}^n_{tt}\nabla\alpha,\nabla \bar{\psi}^n_{tt})_{L^2} \textup{d}s\\
&+\int_0^t \left(-\tau Ng_{ttt}-\alpha Ng_{tt}+c^2 \Delta Ng+b\Delta Ng_t , -\Delta \bar{\psi}_{tt}^n\right)_{L^2}\, \textup{d}s\\
&-c^2 \left(-\Delta \bar{\psi}^n(t),-\Delta \bar{\psi}^n_{t}(t)\right)_{L^2}+c^2\int_0^t\left(-\Delta \bar{\psi}^n_{t},-\Delta \bar{\psi}^n_{t}\right)_{L^2}\, \textup{d}s\\
&+\left(f(t), -\Delta \bar{\psi}_t^n(t) \right)_{L^2}-\int_0^t \left(f_t,-\Delta \bar{\psi}_t^n \right)_{L^2}\, \textup{d}s.
\end{aligned}
\end{equation}
\indent Compared to the higher-regularity result with Dirichlet data~\cite[Theorem 4.1]{kaltenbacher2019jordan}, the main difference in deriving the energy estimates arises due to the appearance of integrals involving the extension of the inhomogeneous boundary data. We can estimate these terms in \eqref{enid1} as follows 
\begin{equation*}
\begin{aligned}
&\int_0^t \int_{\Gamma}(\tau Ng_{ttt}+\alpha N g_{tt})\, \Delta \bar{\psi}^n_{tt}\, \textup{d}x \textup{d}s\\
=& \, -\int_0^t \int_{\Omega} (\tau  \nabla Ng_{ttt}+\alpha \nabla N g_{tt}+N g_{tt}\nabla \alpha )\cdot \nabla \bar{\psi}^n_{tt} \, \textup{d}x \textup{d}s \\
\leq&\, \tau \|\nabla Ng_{ttt}\|_{L^2L^2} \|\nabla \bar{\psi}^n_{tt}\|_{L^2L^2}+\|\alpha\|_{L^\infty L^\infty} \|\nabla Ng_{tt}\|_{L^2L^2} \|\nabla \bar{\psi}^n_{tt}\|_{L^2L^2}\\
&+\|Ng_{tt}\|_{L^2L^6}\|\nabla \alpha\|_{L^\infty L^3}\|\nabla \bar{\psi}^n_{tt}\|_{L^2L^2}.
\end{aligned}
\end{equation*}
We recall that we can employ the fact that $ N \in \mathcal{L}(H^{-1/2}(\Gamma); H^1(\Omega))$ to further estimate the $N$-terms. Since $g(t)$, $g_t(t)$ $\in H^{1/2}(\Omega)$, \eqref{defN} holds in an $L^2(\Omega)$ sense for $Ng$ and $Ng_t$. We therefore find that
\begin{equation*}
\begin{aligned}
& -\int_0^t \int_{\Omega} (c^2 \Delta Ng+b\Delta Ng_t ) \Delta \bar{\psi}_{tt}^n \, \textup{d}x \textup{d}s\\
=& -\int_0^t \int_{\Omega} (c^2 Ng+b Ng_t ) \Delta \bar{\psi}_{tt}^n \, \textup{d}x \textup{d}s\\
\leq& \, \left(c^2 \|\nabla Ng\|_{L^2L^2}+b \|\nabla Ng_{t}\|_{L^2L^2}\right)\|\nabla \bar{\psi}^n_{tt}\|_{L^2L^2}.
\end{aligned}
\end{equation*}
By applying H\"older's inequality to treat the rest of the terms in \eqref{enid1} and the properties of the mapping $N$, we arrive at the estimate
\begin{equation*}
\begin{aligned}
&\tfrac{\tau}{2} |\nabla \bar{\psi}^n_{tt}(t)|^2+\underline{\alpha}\|\nabla \bar{\psi}^n_{tt}\|^2_{L_t^2L^2}
+\tfrac{b}{2}|-\Delta \bar{\psi}_t^n(t)|^2 \\\smallskip
\leq& \, \|\nabla \alpha\|_{L^\infty L^3}\|\bar{\psi}_{tt}^n\|_{L^2 L^6}\|\nabla \bar{\psi}_{tt}^n\|_{L_t^2L^2} 
+ c^2 |-\Delta \bar{\psi}^n (t)|_{L^2} |-\Delta \bar{\psi}_t^n(t)|_{L^2} \\
&+c^2 \|-\Delta \bar{\psi}^n_t\|_{L_t^2L^2}+\|f\|_{L^\infty L^2}|-\Delta \bar{\psi}_t^n(t)|_{L^2}+\|f_t\|_{L_t^2L^2}\|-\Delta \bar{\psi}_t^n\|_{L_t^2L^2}\\
&+\tau C_{N} \|g_{ttt}\|_{L^2 H^{-1/2}} \|\nabla \bar{\psi}^n_{tt}\|_{L^2L^2}+\|\alpha\|_{L^\infty L^\infty}C_{ N} \|g_{tt}\|_{L^2 H^{-1/2}} \|\nabla \bar{\psi}^n_{tt}\|_{L^2_t L^2}\\
&+C_{H^1, L^6}C_{ N}\|g_{tt}\|_{L^2 H^{-1/2}}\|\nabla \alpha\|_{L^\infty L^3}\|\nabla \bar{\psi}^n_{tt}\|_{L^2_tL^2}\\
&+C_{N}(c^2 \|g\|_{L^2 H^{-1/2}}+b \|g_{t}\|_{L^2 H^{-1/2}})\|\nabla \bar{\psi}^n_{tt}\|_{L^2_t L^2}.
\end{aligned}
\end{equation*}
We further estimate the right-hand side with the help of Young's $\varepsilon$-inequality 
for $\varepsilon \in \{b/8, 1/2, \varepsilon_0 \}$ and the standard embedding results to obtain
\begin{equation*}
\begin{aligned}
&\tfrac{\tau}{2} |\nabla \bar{\psi}^n_{tt}(t)|^2
+\tfrac{b}{2}|-\Delta \bar{\psi}_t^n(t)|^2+\underline{\alpha}\|\nabla \bar{\psi}^n_{tt}\|^2_{L_t^2L^2} \\
\leq& \, C_{H^1, L^6}\|\nabla \alpha \|_{L^\infty L^3}\|\bar{\psi}_{tt}^n\|^2_{L_t^2H^1} \\
&+\tfrac{2c^4}{b} |-\Delta \bar{\psi}^n(t)|^2_{L^2}+\tfrac{b}{8}|-\Delta \bar{\psi}_t^n (t)|^2_{L^2}+c^2 \|-\Delta \bar{\psi}^n_t\|_{L_t^2L^2} \\
&+\tfrac{2}{b} \|f\|^2_{L^\infty L^2}+\tfrac{b}{8}|-\Delta \bar{\psi}_t^n(t) |^2_{L^2}
+\tfrac12 \|f_t\|^2_{L^2L^2}+\tfrac12 \|-\Delta \bar{\psi}_t^n\|^2_{L_t^2L^2}\\
&+5\varepsilon_0 \|\nabla \bar{\psi}^n_{tt}\|^2_{L^2_t L^2}+\tfrac{1}{4 \varepsilon_0} C^2_{N}\left(\tau^2 \|g_{ttt}\|^2_{L^2 H^{-1/2}}+\|\alpha\|^2_{L^\infty L^\infty}\|g_{tt}\|^2_{L^2 H^{-1/2}} \right)\\
&+\tfrac{1}{4 \varepsilon_0}C^2_{N} \left(C_{H^1, L^6}^2\|g_{tt}\|^2_{L^2 H^{-1/2}}\|\nabla \alpha\|^2_{L^\infty L^3}+c^4 \|g\|^2_{L^2 H^{-1/2}}+b^2 \|g_{t}\|^2_{L^2 H^{-1/2}} \right).
\end{aligned}
\end{equation*}
The term $\|-\Delta \bar{\psi}^n(t)\|_{L^2}$ can be bounded as follows
\begin{align} \label{est_Delta_psi} 
\|-\Delta \bar{\psi}^n\|_{L_t^\infty L^2} \leq \sqrt{t} \|-\Delta \bar{\psi}_t^n\|_{L_t^2L^2},
\end{align}
since $\bar{\psi}^n(0)=0$. Altogether, we get
\begin{equation}\label{est_1}
\begin{aligned}
& \tfrac{\tau}{2} |\nabla \bar{\psi}^n_{tt}(t)|_{L^2}^2+(\underline{\alpha}-5 \varepsilon_0)\|\nabla \bar{\psi}^n_{tt}\|^2_{L_t^2L^2}
+\tfrac{b}{4}|-\Delta \bar{\psi}^n_t(t)|_{L^2}^2\\\smallskip
\leq& \, C_{H^1, L^6} \|\nabla \alpha\|_{L^\infty L^3}\|\bar{\psi}_{tt}^n\|^2_{L_t^2H^1} \\
&+ \tfrac{2c^4}{b} T \|-\Delta \bar{\psi}_t^n\|^2_{L_t^2L^2}
+c^2 \|-\Delta \bar{\psi}^n_t\|_{L_t^2L^2} \\
&+\tfrac{2}{b}\|f\|^2_{L^\infty L^2}+\frac12\|f_t\|^2_{L^2L^2}+\frac12\|-\Delta \bar{\psi}_t^n\|^2_{L_t^2L^2}\\
&+\tfrac{1}{4 \varepsilon_0} C^2_{N}\left(\tau^2 \|g_{ttt}\|^2_{L^2 H^{-1/2}}+\|\alpha\|^2_{L^\infty L^\infty}\|g_{tt}\|^2_{L^2 H^{-1/2}} \right)\\
&+\tfrac{1}{4 \varepsilon_0}(C_{H^1, L^6}C_{ N})^2\|g_{tt}\|^2_{L^2 H^{-1/2}}\|\nabla \alpha\|^2_{L^\infty L^3}\\
&+\tfrac{1}{4 \varepsilon_0}C^2_{N}\left(c^4 \|g\|^2_{L^2 H^{-1/2}}+b^2 \|g_{t}\|^2_{L^2 H^{-1/2}}\right).
\end{aligned}
\end{equation}
Note that we need a $\bar{\psi}_{tt}$ term in the $L^2$ spatial norm on the left-hand side in \eqref{est_1} to be able to employ Gronwall's inequality. We thus also have to test our problem with $\bar{\psi}^n_{tt}$ and use an estimate analogous to \eqref{West_first_est_discrete_alphapos}:
\begin{equation} \label{West_first_est_discrete_alphapos1}
\begin{aligned}
&\tau |\bar{\psi}^n_{tt}(t)|^2_{L^2}+\underline{\alpha}\|\bar{\psi}_{tt}\|^2_{L^2L^2}+\tfrac{b}{2}|\nabla \bar{\psi}^n_t(t)|^2_{L^2}+|\bar{\psi}^n_{t}(t)|^2_{L^2}\\
\leq&\, c^2 \left|\int_0^t \int_{\Omega} \nabla \bar{\psi}^n \cdot \nabla \bar{\psi}^n_{tt} \, \textup{d}x \textup{d}s \right|\\
&+\|\bar{\psi}^n_t\|_{L^2L^2}\|\bar{\psi}^n_{tt}\|_{L^2L^2}+ \|\tilde{f}\|_{L^2L^2} \|\bar{\psi}^n_{tt}\|_{L^2L^2}\,,
\end{aligned}
\end{equation}
where $\tilde{f}=f-\tau Ng_{ttt}-\alpha Ng_{tt}+c^2 Ng+bNg_t$.
Moreover, we have the bound on $\bar{\psi}^n_{ttt}$:
\begin{equation*} 
\begin{aligned}
&\tau \|\bar{\psi}^n_{ttt}\|_{L^2_tL^2}\\
\leq& \, \|\alpha \bar{\psi}^n_{tt}\|_{L^2_tL^2}+c^2\|-\Delta \bar{\psi}^n\|_{L^2_tL^2}+b\|-\Delta \bar{\psi}_t^n\|_{L^2_tL^2}+\|{\tilde{f}}\|_{L^2L^2},
\end{aligned}
\end{equation*}
from which, after also employing \eqref{est_Delta_psi},  we infer that
\begin{equation} \label{est_3}
\begin{aligned}
&\tau^2 \|\bar{\psi}^n_{ttt}\|^2_{L^2_tL^2}\\
\leq& \, 2\|\alpha\|^2_{L^\infty L^\infty}\|\bar{\psi}^n_{tt}\|^2_{L^2_tL^2}+2(c^2 T+b)^2\|-\Delta \bar{\psi}_t^n\|^2_{L^2_tL^2}+2\|{\tilde{f}}\|^2_{L^2L^2}.
\end{aligned}
\end{equation}
We choose $\varepsilon_0=\underline{\alpha}/6$, add \eqref{est_1} and \eqref{est_3} to 
\eqref{West_first_est_discrete_alphapos1}, apply Gronwall's inequality to the resulting estimate, and then take the supremum over $t\in(0,T_n)$, to get the estimate
\begin{equation} \label{1}
\begin{aligned}
& \begin{multlined}[c]\tau^2 \|\bar{\psi}^n_{ttt}\|^2_{L^2L^2}+\tau\|\bar{\psi}^n_{tt}\|^2_{L^\infty L^2}+ \tau \|\nabla \bar{\psi}^n_{tt}\|^2_{L^\infty L^2}\\
+\|\bar{\psi}^n_{tt}\|^2_{L^2 H^1}+\|- \Delta \bar{\psi}_t^n\|^2_{L^\infty L^2} \end{multlined}\\[1mm]
\leq&\,C(\alpha, T)  \begin{multlined}[t](\tau^2 \|g_{ttt}\|^2_{L^2 H^{-1/2}}+ \|g\|^2_{H^2 H^{-1/2}}+ \| f\|^2_{H^1 L^2}). \end{multlined}
\end{aligned}
\end{equation}
The constant above is given by
\begin{equation*}
\begin{aligned}
C(\alpha, T)= C_5 \,(&\|\alpha\|^2_{L^\infty L^\infty}+\|\nabla \alpha\|^2_{L^\infty L^3}+1)\\
&\times \textup{exp}\,\left(C_6  \left(\|\alpha\|^2_{L^\infty L^\infty}+\|\nabla \alpha\|^2_{L^\infty L^3} +1+T+T^2 \right) T \right)(1+\bar{\tau}),
\end{aligned}
\end{equation*}
where $C_5$, $C_6>0$ do not depend on $n$ or $\tau$. Note that the factor $1+\bar{\tau}$ in the constant above comes from the $\tau$ dependence of $b$ according to \eqref{b}. Since the right-hand side of \eqref{1} does not depend on $T_n$, we are allowed to extend the existence interval to $(0,T)$. \\
\indent Thanks to the derived estimate \eqref{1}, there exists a subsequence, denoted again by $\{\bar{\psi}^n\}_{n \in \mathbb{N}}$, and a function $\bar{\psi}$ such that
\begin{equation*} 
\begin{alignedat}{4} 
\bar{\psi}_{ttt}^n  &\relbar\joinrel\rightharpoonup \bar{\psi}_{ttt} &&\text{ weakly}  &&\text{ in } &&L^2(0,T;L^2(\Omega)),  \\
\bar{\psi}_{tt}^n  &\relbar\joinrel\rightharpoonup \bar{\psi}_{tt} &&\text{ weakly-$\star$}  &&\text{ in } &&L^\infty(0,T;H^1(\Omega)),  \\
\bar{\psi}_t^n &\relbar\joinrel\rightharpoonup \bar{\psi}_t &&\text{ weakly-$\star$} &&\text{ in } &&L^\infty(0,T; H^2(\Omega)).
\end{alignedat} 
\end{equation*}
It can be shown analogously to~\cite[Theorem 4.1]{kaltenbacher2019jordan} that $\bar{\psi} \in X^{W}$ solves \eqref{ibvp_linear_hom} and that estimate \eqref{1} holds with $\bar{\psi}^n$ replaced by $\bar{\psi}$. \\
\indent We then obtain $\psi=\bar{\psi}+Ng$ as the solution to \eqref{ibvp_linear} that satisfies \eqref{energy_est_lin_higher}. Note that we can conclude that $\psi=\bar{\psi}+Ng \in W^{1,\infty}(0,T;H^{2}(\Omega))$ since we assumed that $g \in W^{1,\infty}(0,T;H^{1/2}(\Gamma))$.
\end{proof}
\section{Existence of solutions for the relaxed JMGT equation} \label{sec:West_relaxed_fixedpoint}
We next show existence of solutions for the relaxed JMGT equation \eqref{WesterveltMC_relaxed} with $\tau>0$ by relying on Schauder's fixed-point theorem.

\begin{theorem}\label{thm:relaxed_JMGT}
	Let $c^2$, $b$, $k$, $\tau>0$, and let $T>0$. Assume that the function $h \in C^0(\mathbb{R})$ satisfies
	\begin{align}\tag{\ref{conditions_h}}
	\underline{\alpha} \leq h(s) \leq \overline{\alpha}, \quad \forall s \in \mathbb{R},
	\end{align}
	 and that $g \in H^2(0,T; H^{-1/2}(\Gamma))$ with $(g, g_t)\vert_{t=0}=(0, 0)$.
 Moreover, let
	\begin{equation*}
	\begin{aligned}
	\|g\|^2_{W^{1, \infty}H^{-1/2 }}+\|g_{t}\|^2_{H^1 H^{-1/2}} \leq \varrho.
	\end{aligned}
	\end{equation*}
Then for sufficiently small $\varrho$, there exists a solution $\psi$ of the problem
	\begin{equation} \label{ibvp_nonlinear_relaxed}
	\begin{aligned}
	\begin{cases}
	\tau\psi_{ttt}+h(\psi_t)\psi_{tt}-c^2\Delta \psi - b\Delta \psi_t = 0 \quad \mbox{ in }\Omega\times(0,T), \\[2mm]
	\dfrac{\partial \psi}{\partial n}=g \quad \mbox{ on } \Gamma\times(0,T),\\[2mm]
	(\psi, \psi_t, \psi_{tt})=(0, 0, 0) \quad \mbox{ in }\Omega\times \{0\},
	\end{cases}
	\end{aligned}
	\end{equation}
	in the weak $(H^{1})^{\star}$ sense such that
	\begin{equation*}
	\begin{aligned}
	\psi \in \, X= W^{1, \infty}(0;T;H^1(\Omega)) \cap W^{2, \infty}(0,T; L^2(\Omega)) \cap H^3(0,T;H^1(\Omega)^\star),
	\end{aligned}
	\end{equation*}
and the following estimate holds
	\begin{equation}\label{energy_est_nl_relaxed}
	\begin{aligned}
& \tau^2 \|\psi_{ttt}\|_{L^2 (H^1)^*}^2+\tau \|\psi_{tt}\|^2_{L^\infty L^2}+\|\psi_t\|^2_{L^\infty H^1}\\[1mm]
\leq&\, C(\tau, T)(\|g\|^2_{W^{1, \infty}H^{-1/2 }}+\|g_{t}\|^2_{H^1 H^{-1/2}}). 
	\end{aligned}
	\end{equation}
\end{theorem}
\begin{proof}
We introduce the mapping $\mathcal{F}:v \mapsto \psi$, where $v \in \mathcal{B}$
\begin{equation*}
\begin{aligned}
\mathcal{B}=\{v \in \, X : \, \tau^2 &\|v_{ttt}\|_{L^2 (H^1)^\star}^2+\tau \|v_{tt}\|^2_{L^\infty L^2}+\|v_t\|^2_{L^\infty H^1}\leq M \\[1mm]
& (v, v_t, v_{ttt})\vert_{t=0}=(0, 0, 0) \},
\end{aligned}
\end{equation*}
and $\psi$ solves
\begin{equation} \label{relaxed_linearization}
\begin{aligned}
\tau\psi_{ttt}+h(v)\psi_{tt}-c^2\Delta \psi - b\Delta \psi_t = 0 \mbox{ in }\Omega\times(0,T),
\end{aligned}
\end{equation}
in the weak sense with inhomogeneous Neumann conditions and zero initial conditions. We note that the set $\mathcal{B}$ is non-empty, weakly$-\star$ compact, and convex, and that the mapping $\mathcal{F}$ is well-defined thanks to Theorem~\ref{th:wellposedness_lin_lower_tau}.\\
\indent We can achieve that $\mathcal{F}(\mathcal{B}) \subset \mathcal{B}$ for sufficiently small $\varrho$. Indeed, let $v \in \mathcal{B}$. Then, on account of Theorem~\ref{th:wellposedness_lin_lower_tau} and estimate \eqref{energy_est_lin} for $f=0$, we know that
	\begin{equation*}
	\begin{aligned}
& \tau^2 \|\psi_{ttt}\|_{L^2 (H^1)^\star}^2+\tau \|\psi_{tt}\|^2_{L^\infty L^2}+\|\psi_t\|^2_{L^\infty H^1}\\[1mm]
\leq&\, C(\tau, T)(\|g\|^2_{W^{1, \infty}H^{-1/2 }}+\|g_{t}\|^2_{H^1 H^{-1/2}}). 
	\end{aligned}
	\end{equation*}
From here it follows that $\psi \in \mathcal{B}$ when $\varrho$ is sufficiently small so that $C(\tau, T) \varrho \leq M$ holds. \\[1mm]
\noindent \textbf{Weak$^\star$ continuity}. We want to show that $\mathcal{F}: \mathcal{B} \rightarrow \mathcal{B}$ is weak$^\star$ continuous. Let $\{v^{n}\}_{n \in \mathbb{N}} \subset \mathcal{B}$  be a sequence that weakly$^\star$ converges to $v$ in $X$. Denote $\psi^n =\mathcal{F}(v^n) \in \mathcal{B}$ and $\psi=\mathcal{F}(v) \in \mathcal{B}$. Thanks to the uniform bound provided by Theorem~\ref{th:wellposedness_lin_lower_tau} and standard compactness results, there exists a subsequence, that we do not relabel, and a function $\varphi \in \mathcal{B}$ such that 
\begin{equation*} 
\begin{alignedat}{4} 
\psi_{ttt}^n  &\relbar\joinrel\rightharpoonup \varphi_{ttt} &&\text{ weakly}  &&\text{ in } &&L^2(0,T; (H^1(\Omega))^\star),  \\
\psi_{tt}^n  &\relbar\joinrel\rightharpoonup \varphi_{tt} &&\text{ weakly-$\star$}  &&\text{ in } &&L^\infty(0,T;L^2(\Omega)),  \\
\psi_t^n &\relbar\joinrel\rightharpoonup \varphi_t &&\text{ weakly-$\star$} &&\text{ in } &&L^\infty(0,T; H^1(\Omega)).
\end{alignedat} 
\end{equation*}
Note that by continuity of $h$, we have $h(v^n) \rightarrow h(v)$ a.e. in $\Omega \times (0,T)$. It is then straightforward to check that $\varphi$ solves \eqref{relaxed_linearization}, from which it follows that $\varphi=\psi$ since $\psi$ is the unique solution. We then conclude by a subsequence-subsequence argument that $\{\mathcal{F}(v^n)\}_{n \in \mathbb{N}}$ converges weakly$-\star$ to $\psi$.
\\
\indent The statement now follows by employing Schauder's fixed-point theorem; cf. \cite{Fan1952}.
\end{proof}
\begin{remark}\label{rem:contractivity}
\indent Let $v^{(1)}, v^{(2)} \in \mathcal{B}$ and $\psi^{(1)}=\mathcal{F}(v^{(1)})$, $\psi^{(2)}=\mathcal{F}(v^{(2)}) \in \mathcal{B}$. The difference $\psi=\psi^{(1)}-\psi^{(2)}$ then solves
\begin{equation} \label{contractivity_lower}
\begin{aligned}
\tau \psi_{ttt}+h(v_t^{(1)})\psi_{tt}-c^2 \Delta \psi-b \Delta \psi_t=-(h(v_t^{(1)})-h(v_t^{(2)}))\psi^{(2)}_{tt}
\end{aligned}
\end{equation}
in the weak sense with zero boundary and initial data. To show contractivity of the mapping $\mathcal{F}$, we would need higher regularity of solutions that would -- together with Lipschitz continuity of $h$ with constant $L$ -- allow for the right hand side of \eqref{contractivity_lower} to be estimated as
\begin{equation} \label{contractivity_lower_estrhs}
\|(h(v_t^{(1)})-h(v_t^{(2)}))\psi^{(2)}_{tt}\|_{L^2L^2} \leq L
\|v_t^{(1)}-v_t^{(2)}\|_{L^\infty L^4}\|\psi^{(2)}_{tt}\|_{L^2L^4}.
\end{equation}
This is not possible with the lower order energy estimate from Theorem \ref{th:wellposedness_lin_lower}, but will be enabled by Theorem \ref{th:wellposedness_lin_higher} in the next section.
\end{remark}
\begin{remark}
In case the condition \eqref{conditions_h} is replaced by a non-degeneracy condition:
	\begin{align}
	0 < \underline{\alpha} \leq h(s) \leq \overline{\alpha}, \quad \forall s \in \mathbb{R},
	\end{align}
it can be shown by relying on Theorem~\ref{th:wellposedness_lin_lower} that the bound \eqref{energy_est_lin} is uniform with respect to $\tau$.	
\end{remark}

\section{Well-posedness of the JMGT equation} \label{sec:West_fixedpoint}
Based on the higher-order regularity result of Theorem \ref{th:wellposedness_lin_higher}, we can now use a contraction principle to prove well-posedness of the JMGT equation \eqref{WesterveltMC} with $\tau>0$, as well as an energy bound that is uniform in $\tau$.

\begin{theorem}\label{thm:JMGT}
	Let $c^2$, $b$, $\tau>0$, and let $T>0$, $k \in \mathbb{R}$. 
    Assume that $g \in H^3(0,T; H^{-1/2}(\Gamma))\cap W^{1,\infty}(0,T;H^{1/2}(\Gamma))$ with
    $(g, g_t , g_{tt})\vert_{t=0}=(0, 0, 0)$, and that
	\begin{equation*}
	\begin{aligned}
    \|g\|^2_{W^{1,\infty} H^{1/2}}+\|g\|^2_{H^2H^{-1/2}}+\tau \|g_{tt}\|^2_{L^\infty H^{-1/2}}+\tau^2\|g_{ttt}\|^2_{L^2 H^{-1/2}}\leq \varrho.
	\end{aligned}
	\end{equation*}
Then for sufficiently small $\varrho$, there exists a unique solution $\psi$ of the problem
	\begin{equation} \label{ibvp_nonlinear}
	\begin{aligned}
	\begin{cases}
	\tau\psi_{ttt}+(1-2k\psi_t)\psi_{tt}-c^2\Delta \psi - b\Delta \psi_t = 0 \quad \mbox{ in }\Omega\times(0,T), \\[2mm]
	\dfrac{\partial \psi}{\partial n}=g \quad \mbox{ on } \Gamma\times(0,T),\\[2mm]
	(\psi, \psi_t, \psi_{tt})=(0, 0, 0) \quad \mbox{ in }\Omega\times \{0\},
	\end{cases}
	\end{aligned}
	\end{equation}
	in the strong $L^2$ sense that satisfies 
	\begin{equation*} 
	\begin{aligned}
	\psi \in \, X= W^{1, \infty}(0;T;H^2(\Omega)) \cap W^{2, \infty}(0,T; H^1(\Omega))
\cap H^3(0,T;L^2(\Omega))
	\end{aligned}
	\end{equation*}
and the estimate
\begin{equation} \label{energy_est_nl_higher}
\begin{aligned}
& \tau^2 \|\psi_{ttt}\|_{L^2L^2}+\tau\|\psi_{tt}\|^2_{L^\infty H^1}
+\|\psi_{tt}\|^2_{L^2 H^1}+\|- \Delta \psi_t\|^2_{L^\infty L^2}\\[1mm]
\leq&\,C(T)  \begin{multlined}[t]\left(\|g_t\|^2_{L^{\infty} H^{1/2}}+\|g\|^2_{H^2H^{-1/2}}+\tau \|g_{tt}\|^2_{L^\infty H^{-1/2}}\right.\\ \left. +\tau^2\|g_{ttt}\|^2_{L^2 H^{-1/2}} \right).\end{multlined}
\end{aligned}
\end{equation}
\end{theorem}
\begin{proof}
The proof goes along the lines of the proof of Theorem \ref{thm:relaxed_JMGT}, with the obvious modifications of topologies according to the stronger energies enabled by Theorem \ref{th:wellposedness_lin_higher}, as well as a contraction argument in place of Schauder's fixed-point theorem. \\
\indent We again use the fixed-point operator $\mathcal{F}$ from the proof of Theorem \ref{thm:relaxed_JMGT} with the particular choice $h(z)=1-2kz$ and show that it is a self-mapping on the set
\begin{equation}
\begin{aligned}
\mathcal{B}=\{v \in \, X : \, &\!\begin{multlined}[t]\tau^2 \|v_{ttt}\|_{L^2 L^2}^2+\tau \|v_{tt}\|^2_{L^\infty H^1}\\+\|v_{tt}\|^2_{L^2 H^1}+\|-\Delta v_t\|^2_{L^\infty L^2}\leq M,\end{multlined} \\[1mm]
& (v, v_t, v_{ttt})\vert_{t=0}=(0, 0, 0) \},
\end{aligned}
\end{equation}
with $M$ chosen appropriately, provided that $\varrho$ is sufficiently small.\\[2mm]
\noindent \textbf{\mathversion{bold}$\mathcal{F}$ is a self-mapping}.  For proving that $\mathcal{F}$ is a self-mapping, we use Theorem~\ref{th:wellposedness_lin_higher} in place of Theorem~\ref{th:wellposedness_lin_lower_tau}, where we choose $\alpha=1-2kv_t$ for $v\in \mathcal{B}$. This additionally requires to prove smallness of $\|1-\alpha\|_{L^\infty L^\infty}$ in order to establish non-degeneracy with a uniform constant $\ul{\alpha}$ and of $\|\nabla \alpha\|_{L^\infty L^3}$. We first note that
\begin{equation*} 
\begin{aligned}
\|v_t\|_{L^\infty H^2}&
\leq \|(-\Delta+\mbox{id})^{-1}\|_{L^2\to H^2} \Bigl(\sqrt{T}\|v_{tt}\|_{L^2 H^1}+\|-\Delta v_t\|_{L^\infty L^2}\Bigr)\\
&\leq C(T,\Omega) \sqrt{\|v_{tt}\|_{L^2 H^1}^2+\|-\Delta v_t\|_{L^\infty L^2}^2}\,,
\end{aligned}
\end{equation*}
where $(-\Delta+\mbox{id})$ is equipped with homogeneous Neumann boundary conditions on $\Gamma$. In other words, for $v\in L^2(\Omega)$, $z=(-\Delta+\mbox{id})^{-1}v$ solves
\begin{equation*}
\begin{aligned}
-\Delta z+z=&\, v \ \mbox{ in }\Omega,\\
\frac{\partial z}{\partial n}=& \, 0 \ \mbox{ on }\Gamma\,.
\end{aligned}
\end{equation*}
Therefore, it holds that
\[
\begin{aligned}
\|1-\alpha\|_{L^\infty L^\infty}=2|k|\, \|v_t\|_{L^\infty L^\infty}\leq 2|k| C_{H^2,L^\infty} C(T,\Omega) \sqrt{M}\,,\\
\|\nabla\alpha\|_{L^\infty L^3}=2|k|\, \|\nabla v_t\|_{L^\infty L^3}\leq 2|k| C_{H^1,L^3} C(T,\Omega) \sqrt{M}.
\end{aligned}
\]
Using energy estimate \eqref{energy_est_lin_higher} for the linearized JMGT equation with $f=0$ and choosing $\varrho$ and $M$ sufficiently small yields $\mathcal{F}(v)\in\mathcal{B}$.  \\[2mm]
\noindent \textbf{\mathversion{bold}$\mathcal{F}$ is contractive.} For proving contractivity, we can directly make use of estimate \eqref{contractivity_lower_estrhs} in Remark \ref{rem:contractivity} with $L=2|k|$, and the result on $H^1$ regularity with $\tau$ independent energy bound Theorem \ref{th:wellposedness_lin_lower}, as well as the fact that by the already shown self-mapping property of $\mathcal{F}$, we have that $\psi^{(2)}=\mathcal{F}(v^{(2)})\in\mathcal{B}$. This provides us with the bound 
$\|\psi^{(2)}_{tt}\|_{L^2L^4}\leq C_{H^1,L^4}^\Omega \sqrt{M}$, which by possibly decreasing $M$ yields contractivity of $\mathcal{F}$ in the norm induced by the energy of Theorem \ref{th:wellposedness_lin_lower}:
$$|||v|||:=\sqrt{\tau^2 \|v_{ttt}\|_{L^2 (H^1)^*}^2+\tau \|v_{tt}\|^2_{L^\infty L^2}+\|v_{tt}\|^2_{L^2L^2}+\|v_t\|^2_{L^\infty H^1}}.$$
~\\[-1mm]
\noindent \textbf{\mathversion{bold}$\mathcal{B}$ is closed.} Closedness of $\mathcal{B}$ with respect to this norm can be seen as follows. For any sequence $(\psi_k)_k\in\mathbb{N}\subseteq\mathcal{B}$ converging with respect to $|||\cdot|||$ with limit $\psi$, we have $$||||\psi_k||||^2:=\tau^2 \|v_{ttt}\|_{L^2 L^2}^2+\tau \|v_{tt}\|^2_{L^\infty H^1}+\|v_{tt}\|^2_{L^2 H^1}+\|-\Delta v_t\|^2_{L^\infty L^2}\leq M.$$ Indeed, due to the imposed homogeneous initial conditions, $||||\cdot||||$ defines a norm equivalent to the norm on $\tilde{X}:=H^3(0,T;L^2(\Omega))\cap W^{2,\infty}(0,T;H^1(\Omega))\cap W^{1,\infty}(0,T;H^2(\Omega))$, which is the dual of a separable space. Hence $(\psi_k)_k\in\mathbb{N}$ has a subsequence that converges in the weak* topology of $\tilde{X}$ to some $\bar{\psi}$ that by weak* semicontinuity of the norm lies in $\mathcal{B}$. By uniqueness of limits $\psi$ has to coincide with $\bar{\psi}$ and therefore lies in $\mathcal{B}$.

Altogether, this yields unique existence of a fixed point of $\mathcal{F}$, i.e., of a solution to \eqref{ibvp_nonlinear} in $\mathcal{B}$.
\end{proof}
\begin{remark}
Compared to \cite{KLP12_JordanMooreGibson}, where a pressure formulation of the JMGT  
\[
\tau p_{ttt}+p_{tt}-c^2\Delta p - b\Delta p_t = \tilde{k}(p)^2_{tt}, 
\]
along with homogeneous Dirichlet boundary conditions is considered, and also results on global existence and exponential decay are provided, we here focus on local in time well-posedness only, but extend the setting to inhomogeneous Neumann and absorbing boundary conditions. Due to the differences in formulation (pressure versus velocity potential) and energy estimates also the outcome of the local results in \cite{KLP12_JordanMooreGibson} and Theorem \ref{thm:JMGT} here slightly differ. \cite[Theorem 1.4]{KLP12_JordanMooreGibson} states $p\in W^{2,\infty}(0,T;L^2(\Omega))\cap W^{1,\infty}(0,T;H_0^1(\Omega))\cap L^\infty(0,T;H^2(\Omega)\cap H_0^1(\Omega))$ for sufficiently small initial data $(p(0),p_t(0),p_{tt}(0))\in (H^2(\Omega)\cap H_0^1(\Omega)) \times H_0^1(\Omega)\times L^2(\Omega)$.
\end{remark}

\section{Singular limit for vanishing relaxation time} \label{sec:Singular_limit}

We now study the limiting behavior of solutions $\psi^\tau$ to the JMGT equation \eqref{WesterveltMC} as the relexation time $\tau$ tends to zero. Our goal is to prove convergence in a certain sense to a solution $\bar{\psi}$ of the Westervelt equation \eqref{Westervelt}.\\
\indent A crucial prerequisite for this purpose is the fact that the energy estimate in Theorem \eqref{thm:JMGT} holds uniformly with respect to $\tau$ and that the bound $\rho$ on the data can be chosen independently of $\tau\in(0,\bar{\tau}]$ for any fixed $\bar{\tau}>0$. This will provide us with a uniform bound for the $\tau$-independent part of the energy. In other words, we will derive a uniform bound on $\|\psi^{\tau}\|_{\bar{X}^W}$, where 
\[
\bar{X}^W=\{v\in H^2(0,T;H^1(\Omega))\cap W^{1,\infty}(0,T;H^2(\Omega)): \, v(0)=0, \, v_t(0)=0\}.
\]
\indent Note that the initial conditions imposed in the definition of $\bar{X}^W$ are well-defined in an $H^2(\Omega)$ and $H^1(\Omega)$ sense, respectively, since $\bar{X}^W$ embeds continuously into $C(0,T;H^2(\Omega))\cap C^1(0,T;H^1(\Omega))$.
Therewith, the $\tau$-independent part of the energy defines a norm on $\bar{X}^W$
\[
\|v\|_{\bar{X}^W}=\sqrt{\|v_{tt}\|^2_{L^2 H^1}+\|- \Delta v_t\|^2_{L^\infty L^2}}\,.
\]

\begin{theorem} \label{th:limits}
	Let $c^2$, $b$, $T>0$, $\bar{\tau}>0$, and $k \in \mathbb{R}$. Then there exist $\varrho>0$ such that for all $g \in H^3(0,T; H^{-1/2}(\Gamma))\cap W^{1,\infty}(0,T;H^{1/2}(\Gamma))$ that satisfy $(g, g_t , g_{tt})\vert_{t=0}=(0, 0, 0)$ and 
	\begin{equation*}
	\begin{aligned}
    \|g\|^2_{W^{1,\infty} H^{1/2}}+\|g\|^2_{H^2H^{-1/2}}+\bar{\tau} \|g_{tt}\|^2_{L^\infty H^{-1/2}}+\bar{\tau}^2\|g_{ttt}\|^2_{L^2 H^{-1/2}}\leq \varrho,
	\end{aligned}
	\end{equation*}
    the family $(\psi^\tau)_{\tau\in(0,\bar{\tau})}$ of solutions to \eqref{ibvp_nonlinear} 
	according to Theorem~\ref{thm:JMGT} 
	converges weakly-$\star$ in $\bar{X}^W$ to a solution $\bar{\psi}\in \bar{X}^W$ of \eqref{Westervelt} with homogeneous initial
    conditions $\bar{\psi}(0)=0$, $\bar{\psi}_t(0)=0$, and Neumann boundary conditions $\frac{\partial \bar{\psi}}{\partial n}\vert_{\Gamma}=g$.
\end{theorem}

\begin{proof}
The proof is similar to the one of \cite[Theorem 7.1]{kaltenbacher2019jordan}, but based on different energy estimates.

Uniform boundedness of $\|\psi^\tau\|_{\bar{X}^W}$ according to Theorem~\ref{thm:JMGT} implies existence of a sequence $\tau_\ell\searrow0$, and an element $\bar{\psi}\in \bar{X}^W$ such that $\psi^\ell:=\psi^{\tau_\ell}$ satisfies
\begin{equation*} 
\begin{alignedat}{4} 
\psi_{tt}^\ell  &\relbar\joinrel\rightharpoonup \bar{\psi}_{tt} &&\text{ weakly}  &&\text{ in } &&L^2(0,T; (H^1(\Omega))),  \\
\psi_t^\ell  &\relbar\joinrel\rightharpoonup \bar{\psi}_{tt} &&\text{ weakly-$\star$}  &&\text{ in } &&L^\infty(0,T;H^2(\Omega)),  \\
\psi_t^\ell &\relbar\joinrel\to \bar{\psi}_t &&\text{ strongly } &&\text{ in } &&L^\infty(0,T; L^4(\Omega)),\\
\frac{\partial \psi^\ell}{\partial n} \Bigl \vert _{\Gamma}  & \relbar\joinrel\rightharpoonup \frac{\partial \bar{\psi}}{\partial n} \Bigl \vert _{\Gamma} &&\text{ weakly-$\star$}  &&\text{ in } &&L^2(0,T; (H^{-1/2}(\Gamma)))\cap L^\infty(0,T;H^{1/2}(\Gamma)).
\end{alignedat} 
\end{equation*}
	Therewith, $\frac{\partial \bar{\psi}}{\partial n} \vert _{\Gamma}=g$, and using the fact that $\psi^\ell$ solves \eqref{WesterveltMC}, we get, for $\hat{\psi}_\ell:=\bar{\psi}-\psi^\ell$ and any $v\in C_0^\infty(0,T;C_0^\infty(\Omega))$
	\[
	\begin{aligned}
	&\int_0^T\int_\Omega \left(\bar{\psi}_{tt}-c^2\Delta \bar{\psi} - \delta\Delta \bar{\psi}_t- k(\bar{\psi}_t^2)_t\right)\, v \, \textup{d}x \textup{d}t\\
	=& \,\int_0^T\int_\Omega \left(\hat{\psi}_{\ell\,tt}-c^2\Delta \hat{\psi}_\ell - \delta\Delta \hat{\psi}_{\ell\,t}
	-\tau_\ell\psi^\ell_{ttt} -\tau_\ell c^2\Delta \psi_{\ell\,t}\right)\, v \, \textup{d}x \textup{d}t\\
    &+k\,\int_0^T\int_\Omega \left(\bar{\psi}_t+\psi^\ell_t \right)\hat{\psi}_{\ell\,t} v_t\, \textup{d}x \textup{d}t\\[1mm]
    &\to0 \quad \mbox{ as } \ \ell\to\infty,
	\end{aligned}
	\]
due to the above limits and uniform boundedness of $\psi^\ell$ in $\bar{X}^W$.

A subsequence-subsequence argument, together with uniqueness of the solution to \eqref{Westervelt} according to results in, e.g., \cite{KL09Westervelt,MW11} yields convergence of the whole family $(\psi^\tau)_{\tau\in(0,\bar{\tau})}$.
\end{proof}

\section{Absorbing boundary conditions}\label{sec:ABC}
In this section, we consider extension of our results to the problem with absorbing boundary conditions
	\begin{equation} \label{ibvp_nonlinear_ABC}
	\begin{aligned}
	\begin{cases}
	\tau\psi_{ttt}+(1-2k\psi_t)\psi_{tt}-c^2\Delta \psi - b\Delta \psi_t = 0 \quad \mbox{ in }\Omega\times(0,T), \\[2mm]
	\dfrac{\partial \psi}{\partial n}=g \quad \mbox{ on } \Gamma\times(0,T),\\[2mm]
	\dfrac{\partial \psi}{\partial n}=-\beta\psi_t \quad \mbox{ on } \Sigma\times(0,T),\\[2mm]
	(\psi, \psi_t, \psi_{tt})=(0, 0, 0) \quad \mbox{ in }\Omega\times \{0\},
	\end{cases}
	\end{aligned}
	\end{equation}
where $
\beta>0$. We will comment on all the changes and additions that have to be made and state the corresponding mixed Neumann--absorbing boundary condition versions of the results obtained so far for pure Neumann boundary conditions.

In the proof of well-posedness of the linearized equation with $H^1$ spatial regularity, the corresponding semidiscrete initial-boundary value problem becomes
\begin{equation} \label{ibvp_semi-discrete_ABC}
\begin{aligned} 
\begin{cases}
\ \ (\tau \psi^n_{ttt}+\alpha \psi^n_{tt}, \phi)_{L^2}+(c^2 \nabla \psi^n+b \nabla \psi_t^n, \nabla \phi)_{L^2} \\[1mm]
\ \ +(c^2\beta \psi^n_t+b\beta\psi^n_{tt},\phi)_{L^2(\Sigma)}\\[1mm]
= (f, \phi)_{L^2}+(c^2 g +b g_t, \phi)_{L^2(\Gamma)}, \\[1mm]
\text{for every $\phi \in V_n$ pointwise a.e. in $(0,T)$}, \\[1mm]
(\psi^n(0), \psi_t^n(0), \psi^n_{tt}(0))=(0, 0, 0),
\end{cases}
\end{aligned}
\end{equation}
where the choice of the basis functions $w_i$ is again determined by \eqref{eigenf_Laplacian}; i.e., with homogeneous Neumann conditions only on part of the boundary and no conditions on the rest. This still allows for an orthonormal basis of $L^2(\Omega)$, which is a -- not necessarily orthogonal, but this is not needed -- basis of $H^1(\Omega)$ such that their Dirichlet traces $\mbox{tr}_\Sigma w_i$ form a basis of $L^2(\Sigma)$; cf. \cite{Nikolic15}. 

As a consequence of the fact that absorbing boundary conditions extract energy through the boundary in order to avoid spurious reflections, we get additional energy terms on the left hand side of the energy estimates. More precisely, the terms 
\[
+ c^2\beta\int_0^t |\mbox{tr}_\Sigma\psi^n_{tt}|_{L^2(\Sigma)}^2 \, \textup{d}s
+ b\beta|\mbox{tr}_\Sigma\psi^n_t(t)|_{L^2(\Sigma)}^2
\]
arise in \eqref{West_first_est_discrete}, \eqref{West_2nd_estimate_discrete}, \eqref{West_first_est_discrete_alphapos},
\[
+ c^2\beta\|\mbox{tr}_\Sigma\psi^n_{tt}\|_{L^2 L^2(\Sigma)}^2
+ b\beta\|\mbox{tr}_\Sigma\psi^n_t\|_{L^\infty L^2(\Sigma)}^2
\]
in \eqref{discrete_est_0}, \eqref{discrete_est_1}, and
\[
+ c^2\beta\|\mbox{tr}_\Sigma\psi_{tt}\|_{L^2 L^2(\Sigma)}^2
+ b\beta\|\mbox{tr}_\Sigma\psi_t\|_{L^\infty L^2(\Sigma)}^2
\]
in \eqref{energy_est_lin}, \eqref{energy_est_lin_tau_independent}, \eqref{energy_est_lin_higher}, \eqref{energy_est_nl_relaxed}, \eqref{energy_est_nl_higher}, while the higher order energy identity \eqref{enid1} remains unchanged.\\
\indent Therewith, Theorems~\ref{th:wellposedness_lin_lower_tau}, \ref{th:wellposedness_lin_lower}, and \ref{thm:relaxed_JMGT}
 immediately carry over as follows. 
\begin{theorem} \label{th:wellposedness_lin_lower_tau_ABC}
	Let $c^2$, $b$, $\beta$, $\tau>0$, and let $T>0$. Assume that
	\begin{itemize}
		\item $\alpha \in  L^\infty(0,T; L^\infty(\Omega))$, \smallskip
		\item $f\in L^2(0,T; L^2(\Omega))$, \smallskip
		\item $g \in H^2(0,T; H^{-1/2}(\Gamma))$, 
$\ (g, g_t)\vert_{t=0}=(0, 0)$.
	\end{itemize}
	Then there exists a unique weak solution $\psi$ of the problem 
	\begin{equation}\label{ibvp_linear_ABC}
	\begin{aligned}
	\begin{cases}
	\tau\psi_{ttt}+\alpha(x,t)\psi_{tt}-c^2\Delta \psi - b\Delta \psi_t = f(x,t) \quad \mbox{ in }\Omega\times(0,T), \\[2mm]
	\dfrac{\partial \psi}{\partial n}=g \quad \mbox{ on } \Gamma\times(0,T),\\[2mm]
	\dfrac{\partial \psi}{\partial n}=-\beta\psi_t \quad \mbox{ on } \Sigma\times(0,T),\\[2mm]
	(\psi, \psi_t, \psi_{tt})=(0, 0, 0) \quad \mbox{ in }\Omega\times \{0\},
	\end{cases}
	\end{aligned}
	\end{equation}
	in the weak $(H^{1})^{\star}$ sense that satisfies 
	\begin{equation*}
	\begin{aligned}
	\psi \in \, W^{1, \infty}(0;T;H^1(\Omega)) \cap W^{2, \infty}(0,T; L^2(\Omega))\cap H^3(0,T;H^1(\Omega)^\star).
	\end{aligned}
	\end{equation*}
Furthermore, the solution fullfils the estimate
	\begin{equation*}
	\begin{aligned}
& \begin{multlined}[t] \tau^2 \|\psi_{ttt}\|_{L^2 (H^1)^\star}^2+\tau \|\psi_{tt}\|^2_{L^\infty L^2}+\|\psi_t\|^2_{L^\infty H^1}\\[1mm]
+ c^2\beta\|\textup{tr}\,\psi_{tt}\|_{L^2 L^2(\Sigma)}^2+ b\beta\|\textup{tr}\,\psi_t\|_{L^\infty L^2(\Sigma)}^2 \end{multlined} \\[1mm]
\leq&\, C(\alpha, \tau, T)\left(\|g\|^2_{W^{1, \infty}H^{-1/2 }}+\|g_{t}\|^2_{H^1 H^{-1/2}}+\|f\|^2_{L^2 L^2}\right). 
	\end{aligned}
	\end{equation*}
The constant above is given by
\begin{equation*}
\begin{aligned} 
 & C(\alpha, \tau, T)\\
 =&\, C_1 \left(\tfrac{1}{\tau^2}\|\alpha\|^2_{L^\infty L^\infty}+T^2+1 \right)\textup{exp}(C_2 (\tfrac{1}{\tau}+\tfrac{1}{\tau}\|\alpha\|_{L^\infty L^\infty}+1+T)T)
(1+\tau),
\end{aligned}
\end{equation*}
where $C_1$, $C_2>0$ do not depend on $\tau, T$, or $\alpha$. 
\end{theorem}

\begin{theorem} \label{th:wellposedness_lin_lower_ABC} Let the assumption of Theorem~\ref{th:wellposedness_lin_lower} hold and assume additionally that 
for some fixed $\overline{\tau}>0$, $\tau\in(0,\overline{\tau}]$,
\begin{align} 
\exists \ul{\alpha}>0: \,\alpha(t)\geq\ul{\alpha}\ \text{  a.e. in } \Omega \times (0,T).  
\end{align}
Then the solution of \eqref{ibvp_linear_ABC} satisfies the estimate 
	\begin{equation*}
	\begin{aligned}
&\begin{multlined}[t] \tau^2 \|\psi_{ttt}\|_{L^2 (H^1)^\star}^2+\tau \|\psi_{tt}\|^2_{L^\infty L^2}+\|\psi_{tt}\|^2_{L^2L^2}+\|\psi_t\|^2_{L^\infty H^1}\\[1mm]
+ c^2\beta\|\textup{tr}\,\psi_{tt}\|_{L^2 L^2(\Sigma)}^2+ b\beta\|\textup{tr}\,\psi_t\|_{L^\infty L^2(\Sigma)}^2 \end{multlined} \\[1mm]
\leq&\, C(\alpha, \tau, T)\left(\|g\|^2_{W^{1, \infty}H^{-1/2 }}+\|g_{t}\|^2_{H^1 H^{-1/2}}+\|f\|^2_{L^2 L^2} \right), 
	\end{aligned}
	\end{equation*}
where the constant is given by
\begin{equation*}
\begin{aligned}
C(\alpha, \tau, T)=C_3 \,(1+T^2) \textup{exp}(C_4 (1+T)T)
(1+\overline{\tau}),
\end{aligned}
\end{equation*}
and $C_3$, $C_4>0$ do not depend on $\tau, T$, or $\alpha$.
\end{theorem}
\begin{theorem}\label{thm:relaxed_JMGT_ABC}
	Let $c^2$, $b$, $\beta$, $\tau>0$, $k\in\mathbb{R}$, and let $T>0$. Assume that the function $h \in C^0(\mathbb{R})$ satisfies \eqref{conditions_h}, and that $g \in H^2(0,T; H^{-1/2}(\Gamma))$ with $(g, g_t)\vert_{t=0}=(0,0)$. 
 Moreover, let
	\begin{equation*}
	\begin{aligned}
	\|g\|^2_{W^{1, \infty}H^{-1/2 }}+\|g_{t}\|^2_{H^1 H^{-1/2}} \leq \varrho.
	\end{aligned}
	\end{equation*}
Then for sufficiently small $\varrho$, there exists a solution $\psi$ of the problem
	\begin{equation} 
	\begin{aligned}
	\begin{cases}
	\tau\psi_{ttt}+h(\psi_t)\psi_{tt}-c^2\Delta \psi - b\Delta \psi_t = 0 \quad \mbox{ in }\Omega\times(0,T), \\[2mm]
	\dfrac{\partial \psi}{\partial n}=g \quad \mbox{ on } \Gamma\times(0,T),\\[2mm]
	\dfrac{\partial \psi}{\partial n}=-\beta\psi_t \quad \mbox{ on } \Sigma\times(0,T),\\[2mm]
	(\psi, \psi_t, \psi_{tt})=(0, 0, 0) \quad \mbox{ in }\Omega\times \{0\},
	\end{cases}
	\end{aligned}
	\end{equation}
	in the weak $(H^{1})^{\star}$ sense that satisfies 
	\begin{equation*}
	\begin{aligned}
	\psi \in \,  W^{1, \infty}(0;T;H^1(\Omega)) \cap W^{2, \infty}(0,T; L^2(\Omega)) \cap H^3(0,T;H^1(\Omega)^\star),
	\end{aligned}
	\end{equation*}
and the estimate
	\begin{equation*}
	\begin{aligned}
& \begin{multlined}[t] \tau^2 \|\psi_{ttt}\|_{L^2 (H^1)^\star}^2+\tau \|\psi_{tt}\|^2_{L^\infty L^2}+\|\psi_t\|^2_{L^\infty H^1}\\[1mm]
+ c^2\beta\|\textup{tr}\,\psi_{tt}\|_{L^2 L^2(\Sigma)}^2+ b\beta\|\textup{tr}\,\psi_t\|_{L^\infty L^2(\Sigma)}^2 \end{multlined} \\[1mm]
\leq&\, C(\tau, T)\left(\|g\|^2_{W^{1, \infty}H^{-1/2 }}+\|g_{t}\|^2_{H^1 H^{-1/2}}\right).
	\end{aligned}
	\end{equation*}
\end{theorem}
~\\[4mm]
\noindent \textbf{Extension of the boundary data.} To extend the higher-regularity results of Theorems~\ref{th:wellposedness_lin_higher}, \ref{thm:JMGT}, and \ref{th:limits}, 
we impose the additional compatibility condition $g\vert_{\partial\Gamma}=0$ on the interface between the two boundary parts $\Gamma$ and $\Sigma$; in other words, we assume that $g\in H_0^{1/2}(\Gamma)$. We redefine the extension operator $N$ as $Nh=v$, where $v$ solves
\[
\begin{aligned}
-\Delta v+v&=0 \quad \text{ in } \Omega, \\
\hspace*{1cm}\dfrac{\partial v}{\partial n}&=\tilde{h}=\begin{cases} h \quad \text{ on } \Gamma\\ 0 \quad \text{ on } \partial\Omega\setminus\Gamma\,,
\end{cases}
\end{aligned}
\]
with $\tilde{h}\in H^s(\partial\Omega)$ for all $s\in[0,\tfrac12)$, provided that $h\in H_0^{1/2}(\Gamma)$; cf. \cite[Corollary 1.4.4.5.]{Grisvard}. Thus we still have boundedness of $N$ as an operator $H^{-1/2}(\Gamma)\to H^1(\Omega)$ and as $H^{1/2}(\Gamma)\to H^{3/2+s}(\Omega)$ for $s\in[0,\tfrac12)$.\\
\indent On the other hand, we will also need an $L^2(\Omega)$ estimate on $Ng_{tttt}$. Therefore, we define $v=Nh$ by duality for $h\in H^{-3/2}$, i.e., 
\[
\begin{aligned} 
(\nabla v , -\Delta \phi +\phi) = \langle h,\phi\rangle_{H^{-3/2}(\Gamma),H^{3/2}(\Gamma)},
\end{aligned}
\]
for every $\phi \in H^2(\Omega)$, $\frac{\partial\phi}{\partial n}\vert_{\partial\Omega}=0$, which yields boundedness of $N:H^{-3/2}(\Gamma)\to L^2(\Omega)$.
\\
\indent A crucial point in the analysis is the fact that in a mixed Neumann--absorbing boundary condition setting, we cannot conclude anymore $H^2$ regularity in space of some function $v$ from $L^2$ boundedness of $-\Delta v+v$.
Nevertheless, we can achieve sufficient regularity of $\psi_t$ to obtain an embedding into $L^\infty(\Omega)$, as required for guaranteeing non-degeneracy, along the lines of the proof of \cite[Theorem 1]{shapederiv}. 
\begin{theorem} \label{th:wellposedness_lin_higher_ABC}
	Let $c^2$, $b>0$, $\tau \in (0, \overline{\tau})$, and let $T>0$. Assume that
	\begin{itemize}
		\item $\alpha \in X_\alpha^W:=L^\infty(0,T; W^{1,3}(\Omega)\cap L^\infty(\Omega))\cap W^{1,\infty}(0,T;L^3(\Omega))$ ,\smallskip
		\item  $\ \exists \, \ul{\alpha}>0: \,\alpha(t)\geq\ul{\alpha}\ \text{  a.e. in } \Omega \times (0,T), $ \smallskip
		\item $f\in H^1(0,T; L^2(\Omega))$, \smallskip
		\item $		\|\nabla \alpha\|_{L^\infty L^3}<\underline{\alpha}/\left (6  C_{H^1,L^6} \right)$, \smallskip
  	\item 
$g \in H^4(0,T; H^{-3/2}(\Gamma))\cap H^3(0,T; H^{-1/2}(\Gamma))\cap H^2(0,T;H_0^{1/2}(\Gamma))$,
\smallskip
    \item $(g, g_t , g_{tt})\vert_{t=0}=(0, 0, 0)$.
	\end{itemize}
	Then there exists a unique weak solution $\psi$ of the problem  \eqref{ibvp_linear_ABC} that satisfies
		\begin{equation*}
	\begin{aligned}
	\psi \in \,  W^{1, \infty}(0,T; H^{3/2+s}(\Omega)) \cap W^{2, \infty}(0,T;H^1(\Omega))\cap H^3(0,T;L^2(\Omega))
	\end{aligned}
	\end{equation*}
		for any $s\in[0,\frac12)$. 	Moreover, the solution fulfills the estimate
\begin{equation*}
\begin{aligned}
& \begin{multlined}[c]\tau^2 \|\psi_{ttt}\|_{L^2L^2}+\tau\|\psi_{tt}\|^2_{L^\infty H^1}
\\
+\|\psi_{tt}\|^2_{L^2 H^1}+\|- \Delta \psi_t\|^2_{L^\infty L^2}\\ 
+ c^2\beta\|\textup{tr}\,\psi_{tt}\|_{L^2 L^2(\Sigma)}^2+ b\beta\|\textup{tr}\,\psi_t\|_{L^\infty L^2(\Sigma)}^2
\end{multlined}\\[1mm]
\leq&
\,C(\alpha, T)  \begin{multlined}[t]
\left(\|g_t\|^2_{H^1 H^{1/2}}+\|g\|^2_{H^3H^{-1/2}}+\tau^2\|g_{ttt}\|^2_{H^1 H^{-1/2}} \right.
\\\left.+\tau^2 \|g_{tttt}\|^2_{L^2 H^{-3/2}}+ \| f\|^2_{H^1 L^2}\right). \end{multlined}
\end{aligned}
\end{equation*}
The constant above is given by
\begin{equation*}
\begin{aligned}
C(\alpha, T)=&C_5 \,(\|\alpha\|^2_{L^\infty L^\infty}+\|\nabla \alpha\|^2_{L^\infty L^3}+1)\\
&\times \textup{exp}\,\left(C_6 \, \left(\|\alpha\|^2_{L^\infty L^\infty}+\|\nabla \alpha\|^2_{L^\infty L^3} +1+T+T^2 \right) T  \right)
(1+\overline{\tau}),
\end{aligned}
\end{equation*}
where $C_5$, $C_6>0$ do not depend on $n$ or $\tau$. 
\end{theorem}
\begin{proof}
We highlight here the main differences to the proof of Theorem~\ref{th:wellposedness_lin_higher}, which consist in obtaining an energy estimate that provides us with enough regularity of $\mbox{tr}_\Sigma \psi_{tt}$. We choose $\{w_i\}_{i \in \mathbb{N}}$ again as eigenfunctions of the homogeneous Neumann-Laplacian \eqref{eigenf_Laplacian}  and consider Galerkin approximations in $V_n=\textup{span}\{w_1, \ldots w_n\}$ of the 
difference $\bar{\psi}=\psi-Ng$. We thus obtain the semi-discrete problem
\begin{equation}
\begin{aligned} \label{ibvp_abs}
\begin{cases}
\ \ (\tau \bar{\psi}^n_{ttt}+\alpha \bar{\psi}^n_{tt}, \phi)_{L^2}+(c^2 \nabla \bar{\psi}^n+b \nabla \bar{\psi}_{t}^n, \nabla \phi)_{L^2} \\[0.5mm]
\ \ +(c^2\beta \bar{\psi}^n_t+b\beta\bar{\psi}^n_{tt},\phi)_{L^2(\Sigma)}\\[0.5mm]
= (\tilde{f}, \phi)_{L^2}-(c^2 \beta Ng_t+b \beta Ng_{tt}, \phi)_{L^2(\Sigma)}, \\[0.5mm]
\text{for every $\phi \in V_n$ pointwise a.e. in $(0,T)$},\\[1mm]
(\bar{\psi}^n(0), \bar{\psi}_t^n(0), \bar{\psi}^n_{tt}(0))=(0, 0, 0),
\end{cases}
\end{aligned}
\end{equation}
where $\tilde{f}=f-\tau Ng_{ttt}-\alpha Ng_{tt}+c^2 \Delta Ng+b \Delta Ng_{t}$. \\[2mm]
\noindent \textbf{Higher-order estimate.}  Due to the regularity assumptions on $g$, $f$ and $\alpha$, we can conclude that problem \eqref{ibvp_abs} has a solution $\bar{\psi}^n \in H^4(0,T; V_n)$. We are thus allowed to differentiate \eqref{ibvp_abs} with respect to time and also consider the following problem
\begin{equation}  \label{ibvp_abs_t}
\begin{aligned} 
\begin{cases}
\ \ (\tau \bar{\psi}^n_{tttt}+\alpha \bar{\psi}^n_{ttt}, \phi)_{L^2}+(c^2 \nabla \bar{\psi}^n_t+b \nabla \bar{\psi}_{tt}^n +\alpha_t\bar{\psi}^n_{tt}, \nabla \phi)_{L^2} \\[1mm]
\ \ +(c^2\beta \bar{\psi}^n_{tt}+b\beta\bar{\psi}^n_{ttt},\phi)_{L^2(\Sigma)}\\[1mm]
= (\tilde{f}_t, \phi)_{L^2}-(c^2 \beta Ng_{tt}+b \beta Ng_{ttt}, \phi)_{L^2(\Sigma)}, \\[1mm]
\text{for every $\phi \in V_n$ pointwise a.e. in $(0,T)$}, 
\end{cases}
\end{aligned}
\end{equation}
where $\tilde{f}_t=f_t-\tau Ng_{tttt}-\alpha Ng_{ttt}-\alpha_t Ng_{tt}+c^2 \Delta Ng_t+b \Delta Ng_{tt}$. Above we have used the fact that $\frac{\partial Ng}{\partial n}$ vanishes on $\Sigma$. We note first that by testing \eqref{ibvp_abs} with $-\Delta \bar{\psi}_{tt}^n$ and $\bar{\psi}_{tt}^n$, we can derive the following estimate
\begin{equation}
\begin{aligned}
& \begin{multlined}[c]\tau^2 \|\bar{\psi}^n_{ttt}\|^2_{L^2L^2}+\tau\|\bar{\psi}^n_{tt}\|^2_{L^\infty L^2}+ \tau \|\nabla \bar{\psi}^n_{tt}\|^2_{L^\infty L^2}+\|\bar{\psi}^n_{tt}\|^2_{L^2 H^1}\\[1mm]
+\|- \Delta \bar{\psi}_t^n\|^2_{L^\infty L^2}
+ \tau\beta\|\mbox{tr}_\Sigma\bar{\psi}^n_{ttt}\|_{L^2_t L^2(\Sigma)}^2+ \beta|\sqrt{\alpha}\mbox{tr}_\Sigma\psi_{tt}(t)|_{L^2(\Sigma)}^2\end{multlined}\\[1mm]
\leq&\,C(\alpha, T)  \begin{multlined}[t] \left(
\tau^2 \|g_{tttt}\|^2_{L^2 H^{-3/2}}+
\tau^2 \|g_{ttt}\|^2_{L^2 H^{-1/2}}+ \|g\|^2_{H^2 H^{-1/2}}\right. \\ \left. + \| f\|^2_{H^1 L^2} \right), \end{multlined}
\end{aligned}
\end{equation}
where the additional terms on $\Sigma$ arise due to integration by parts of the $\tau$ and $\alpha$ term with respect to space.

We then test \eqref{ibvp_abs_t} with $\psi^n_{ttt}$. After integration by parts of the $c^2$ term, which also removes the $c^2 \beta Ng_t$ term on $\Sigma$,
as well as integration by parts with respect to time of the remaining term on $\Sigma$, we obtain
\[
\begin{aligned}
&\tfrac{\tau}{2} |\bar{\psi}^n_{ttt}(t)|^2_{L^2}+\tfrac{b}{2}|\nabla \bar{\psi}^n_{tt}(t)|^2_{L^2}+\int_0^t|\sqrt{\alpha}\bar{\psi}^n_{ttt}(t)|^2_{L^2}\, \textup{d}s
+ \tfrac{b\beta}{2}|\mbox{tr}_\Sigma\bar{\psi}^n_t(t)|_{L^2(\Sigma)}^2\\[1mm]
=&\, \int_0^t(\bar{\psi}^n_{ttt},\tilde{f}_t-\alpha_t\bar{\psi}^n_{tt}-c^2\Delta\bar{\psi}^n_t)_{L^2}\,\textup{d}s
-\int_0^t (b \beta Ng_{tt}, \bar{\psi}^n_{ttt})_{L^2(\Sigma)}\, \textup{d}s
\\[1mm]
\leq&\tfrac{1}{2}\int_0^t|\sqrt{\alpha}\bar{\psi}^n_{ttt}(t)|^2_{L^2}\, \textup{d}s
+\tfrac{1}{2\underline{\alpha}}\|\tilde{f}_t-\alpha_t\bar{\psi}^n_{tt}-c^2\Delta\bar{\psi}^n_t\|^2_{L^2_t L^2}\\[1mm]
&+\tfrac{b \beta C_{tr}}{2}\left(\|\nabla\bar{\psi}^n_{tt}\|_{L^2_t L^2}^2+\|\bar{\psi}^n_{tt}\|_{L^2_t L^2}^2\right)
+\tfrac{b \beta C_{tr}}{2}\|\mbox{tr}_\Sigma Ng_{ttt}\|_{L^2 H^{-1/2}(\Sigma)}^2\\[1mm]
&+\tfrac{b}{4}\left(|\nabla\bar{\psi}^n_{tt}(t)|_{L^2}^2+|\bar{\psi}^n_{tt}(t)|_{L^2}^2\right)
+b \beta^2 C_{tr}^2|\mbox{tr}_\Sigma Ng_{tt}(t)|_{H^{-1/2}(\Sigma)}^2\,,
\end{aligned}
\]
where the terms $\|\mbox{tr}_\Sigma Ng_{ttt}\|_{L^2 H^{-1/2}(\Sigma)}$ and $|\mbox{tr}_\Sigma Ng_{tt}(t)|_{H^{-1/2}(\Sigma)}$ can be further estimated by means of the mapping properties of $N$, continuity of the embedding $L^2(\Sigma)\to H^{-1/2}(\Sigma)$, and the trace theorem. In the limit as $n\to\infty$, we arrive at the energy estimate
\begin{equation*}
\begin{aligned}
&\tau\|\bar{\psi}_{ttt}\|^2_{L^\infty L^2}+\|\nabla \bar{\psi}_{tt}\|^2_{L^\infty L^2}+\|\bar{\psi}_{ttt}\|^2_{L^2 L^2}+\|\mbox{tr}_\Sigma\bar{\psi}_t\|_{L^\infty L^2(\Sigma)}^2\\[1mm]
\leq& \, C(T) \begin{multlined}[t] \left((1+\|\alpha_t\|^2_{L^\infty L^3})(\|\bar{\psi}_{tt}\|^2_{L^2 H^1}+\|g_{tt}\|_{L^2 H^{-1/2}}^2)
+\|\Delta\bar{\psi}_t\|^2_{L^2 L^2}\right.\\[1mm] \left.
+\|\bar{\psi}_{tt}\|^2_{L^\infty_t L^2}+\|\tilde{f}_t\|^2_{L^2 L^2}+\|g_{ttt}\|_{L^2 H^{-1/2}}^2+\|g_{tt}\|_{L^\infty H^{-1/2}}^2
\right).\end{multlined}
\end{aligned}
\end{equation*}
From here we obtain the estimate for $\psi=\bar{\psi}+Ng$: 
\begin{equation} \label{est_ttt}
\begin{aligned}
&\tau\|\psi_{ttt}\|^2_{L^\infty L^2}+\|\nabla \psi_{tt}\|^2_{L^\infty L^2}+\|\psi_{ttt}\|^2_{L^2 L^2}+\|\mbox{tr}_\Sigma \psi_t\|_{L^\infty L^2(\Sigma)}^2\\[1mm]
\leq& \, C(T) \begin{multlined}[t] \left((1+\|\alpha_t\|^2_{L^\infty L^3})(\|\psi_{tt}\|^2_{L^2 H^1}+\|g_{tt}\|_{L^2 H^{-1/2}}^2) \right.
\\[1mm] 
+\|\Delta \psi_t\|^2_{L^2 L^2}+\|\psi_{tt}\|^2_{L^\infty_t L^2}+\|\tilde{f}_t\|^2_{L^2 L^2}\\[1mm] \left.
+\tau^2\|g_{tttt}\|_{L^2 H^{-3/2}}^2
+\|g_{ttt}\|_{L^2 H^{-1/2}}^2
+\|g_{tt}\|_{L^\infty H^{-1/2}}^2
\right). \end{multlined}
\end{aligned}
\end{equation}
The right-hand side can be further estimated by means of Gronwall's inequalities and the other energy estimates.\\

\noindent \textbf{\mathversion{bold}$H^{3/2+s}(\Omega)$ regularity.} Now we are ready to adopt the argument from the proof of \cite[Theorem 1]{shapederiv} as follows.
Since $z:=-\Delta\psi+\psi$ satisfies the ODE 
\[ 
z_t(t)=-\tfrac{c^2}{b} z(t) +\tfrac{1}{b}(f(t)-\tau\psi_{ttt}(t)-\alpha\psi_{tt}(t)+b\psi_t(t)+c^2\psi(t))
\]
in a pointwise almost every sense with respect to space, we can use the common variation of constants formula to write 
\[
z(t)=\tfrac{1}{b}\int_0^t e^{-\frac{c^2}{b}(t-s)} (f(s)-\tau\psi_{ttt}(s)-\alpha\psi_{tt}(s)+b\psi_t(s)+c^2\psi(s))\, \textup{d}s.
\]
Hence we have that
\[
\begin{aligned}
z_t(t)=&\, \tfrac{1}{b}(f(t)-\tau\psi_{ttt}(t)-\alpha\psi_{tt}(t)+b\psi_t(t)+c^2\psi(t))\\
&-\tfrac{c^2}{b}\int_0^t e^{-\frac{c^2}{b}(t-s)} (f(s)-\tau\psi_{ttt}(s)-\alpha\psi_{tt}(s)+b\psi_t(s)+c^2\psi(s))\, \textup{d}s\\
=:&\, \tilde{f}(t)\,,
\end{aligned}
\]
where $\tilde{f}\in L^\infty(0,T;L^2(\Omega))$, provided $f\in L^\infty(0,T;L^2(\Omega))$.
We now consider this as a pointwise in time elliptic PDE for $\tilde{\psi}:=\psi_t(t)$, equipped with the boundary conditions resulting from \eqref{ibvp_nonlinear_ABC},
\begin{equation}\label{Neumannproblemz}
\begin{aligned}
-\Delta \tilde{\psi}+\tilde{\psi}=&\, \tilde{f}(t) \mbox{ in }\Omega\\
\frac{\partial \tilde{\psi}}{\partial n}=& \, \tilde{g}(t)=\begin{cases}g_t(t) \mbox{ on }\Gamma\\-\beta \psi_{tt}(t) \mbox{ on }\Sigma\,.\end{cases}
\end{aligned}
\end{equation}
\indent We note that the Neumann data $\tilde{g}(t)$ in general is not an element of $H^{1/2}(\partial\Omega)$ even though the functions $g_t(t)$, $\psi_{tt}(t)$ exhibit $H^{1/2}$ regularity on the respective boundary parts, the latter due to the trace theorem and our energy estimate. Global $H^{1/2}$ regularity of the Neumann data would require continuity over the interface between $\Gamma$ and $\Sigma$. Nevertheless, it can be shown (see the appendix of \cite{shapederiv}), that $\tilde{g}(t)$ lies in $H^{s}(\partial\Omega)$ for all $0< s< \frac12$ and that
\[
|\tilde{g}(t)|_{H^s(\partial\Omega)} \leq C_7 \left(|g_t(t)|_{H^s(\Gamma)} + \beta|\mbox{tr}_\Sigma\psi_{tt}(t)|_{H^s(\Sigma)}\right)
 \]
holds;~see \cite[Corollary 1.4.4.5.]{Grisvard} and~\cite[Appendix]{shapederiv}. Hence, elliptic regularity for the Neumann problem \eqref{Neumannproblemz} yields
\[
\begin{aligned}
\|\psi_t\|_{L^\infty H^{3/2+s}} \leq& \, C_8 \|\tilde{g}\|_{L^\infty H^s(\partial\Omega)}\\[1mm]
\leq& \, C_7 C_8 (\|g_t\|_{L^\infty H^s(\Gamma)} + \beta\|\mbox{tr}_\Sigma\psi_{tt}(t)\|_{L^\infty H^s(\Sigma)})\\[1mm]
\leq& \, C_7 C_8 (\|g_t\|_{L^\infty H^s(\Gamma)} + \beta C_{tr} \|\psi_{tt}\|_{L^\infty H^1}),
\end{aligned}
\]
which can be further estimated by the previous and the additional energy estimates.
\end{proof}
By relying on the results of Theorem~\ref{th:wellposedness_lin_lower_ABC}, Theorems~\ref{thm:JMGT} and \ref{th:limits} can be extended in a straightforward manner.  Note that we only need the case $f=0$, $\alpha=1-2k\psi_t$ of estimate \eqref{est_ttt}, since contractivity is already established in a weaker norm. 
\begin{theorem}\label{thm:JMGT_ABC}
	Let $c^2$, $b$, $\beta$, $\tau>0$, $k\in\mathbb{R}$, and let $T>0$. 
    Assume that 
$g \in H^4(0,T; H^{-3/2}(\Gamma))\cap H^3(0,T; H^{-1/2}(\Gamma))\cap H^2(0,T;H_0^{1/2}(\Gamma))$,
    with $(g, g_t, g_{tt})\vert_{t=0}=(0,0, 0)$, and that
	\begin{equation*}
	\begin{aligned}
\|g_t\|^2_{H^1 H^{1/2}}+\|g\|^2_{H^3H^{-1/2}}+\tau^2\|g_{ttt}\|^2_{H^1 H^{-1/2}} 
+\tau^2 \|g_{tttt}\|^2_{L^2 H^{-3/2}}
\leq \varrho.
	\end{aligned}
	\end{equation*}
Then for sufficiently small $\varrho$, there exists a unique solution $\psi$ of \eqref{ibvp_nonlinear_ABC}	in the weak $(H^1)^\star$ sense that satisfies
	\begin{equation*}
	\begin{aligned}
	\psi \in \,  W^{1, \infty}(0;T;H^{3/2+s}(\Omega)) \cap W^{2, \infty}(0,T; H^1(\Omega))
\cap H^3(0,T;L^2(\Omega))
	\end{aligned}
	\end{equation*}
for any $s\in(0,\frac12)$, and the estimate
\begin{equation*}
\begin{aligned}
& 
\tau^2 \|\psi_{ttt}\|_{L^2L^2}+\tau\|\psi_{tt}\|^2_{L^\infty H^1}
+\|\psi_{tt}\|^2_{L^2 H^1}+\|- \Delta \psi_t\|^2_{L^\infty L^2}\\[1mm]
&\qquad+ c^2\beta\|\textup{tr}\,\psi_{tt}\|_{L^2 L^2(\Sigma)}^2+ b\beta\|\textup{tr}\,\psi_t\|_{L^\infty L^2(\Sigma)}^2\\[1mm]
\leq&\,C(T) \begin{multlined}[t]
 \left(\|g_t\|^2_{L^{\infty} H^{1/2}}+\|g\|^2_{H^2H^{-1/2}}+\tau \|g_{tt}\|^2_{L^\infty H^{-1/2}}\right. \\ \left. +\tau^2\|g_{ttt}\|^2_{L^2 H^{-1/2}}
+\tau^2 \|g_{tttt}\|^2_{L^2 H^{-3/2}}\right).\end{multlined}
\end{aligned}
\end{equation*}
\end{theorem}

\begin{theorem} \label{th:limits_ABC}
	Let $c^2$, $b$, $\beta$, $T>0$, $\bar{\tau}>0$, and $k \in \mathbb{R}$. Then there exist $\varrho>0$ such that for all 
$g \in H^4(0,T; H^{-3/2}(\Gamma))\cap H^3(0,T; H^{-1/2}(\Gamma))\cap H^2(0,T;H_0^{1/2}(\Gamma))$,
    that satisfy $(g, g_t, g_{tt})\vert_{t=0}=(0,0, 0)$ and 
	\begin{equation*}
	\begin{aligned}
\|g_t\|^2_{H^1 H^{1/2}}+\|g\|^2_{H^3H^{-1/2}}+\bar{\tau}^2\|g_{ttt}\|^2_{H^1 H^{-1/2}} 
\leq \varrho,
	\end{aligned}
	\end{equation*}
    for any $s\in(0,\frac12)$, the family $(\psi^\tau)_{\tau\in(0,\bar{\tau})}$ of solutions to \eqref{ibvp_nonlinear_ABC} 
	according to Theorem \ref{thm:JMGT_ABC} converges weakly-$\star$ in $$\bar{X}^W=\{v\in H^2(0,T;H^1(\Omega))\cap W^{1,\infty}(0,T;H^{3/2+s}(\Omega)): \, v(0)=0, \, v_t(0)=0\}$$
	 to a solution $\bar{\psi}\in \bar{X}^W$ of \eqref{Westervelt} with homogeneous initial
    conditions $\bar{\psi}(0)=0$, $\bar{\psi}_t(0)=0$, and mixed Neumann -- absorbing boundary conditions $\frac{\partial \bar{\psi}}{\partial n}\vert_{\Gamma}=g$, $\frac{\partial \bar{\psi}}{\partial n}\vert_\Sigma=-\beta\bar{\psi}_t$.
\end{theorem}

\section*{Acknowledgments}
The second author gratefully acknowledges the funding provided by the Deutsche Forschungsgemeinschaft under the grant number WO 671/11-1.
\bibliographystyle{plain}
\bibliography{references}
\end{document}